\documentclass[12pt,leqno,draft]{article} 
\usepackage{amsmath,amssymb,amsthm,amsfonts, indentfirst}
\usepackage{enumerate,color,bm}
\makeatletter
\def\@cite#1#2{[{{\bfseries #1}\if@tempswa , #2\fi}]}
\renewcommand{\section}{%
\@startsection{section}{1}{\z@}
{0.5truecm plus -1ex minus -.2ex}%
{1.0ex plus .2ex}{\bfseries\large}}
\def\@seccntformat#1{\csname the#1\endcsname.\ }
\makeatother

\numberwithin{equation}{section} 
\newtheorem{thm}{Theorem}[section]
\newtheorem{corollary}[thm]{Corollary}
\newtheorem{lem}[thm]{Lemma}
\newtheorem{prop}[thm]{Proposition}
\theoremstyle{definition}

\newtheorem{remark}{Remark}[section] 

\newcommand{\pa}{\partial}
\newcommand{\Rn}{\mathbb{R}^n}

\newcommand{\N}{\mathbb{N}}

\newcommand{\tmax}{T_{\rm max}}
\newcommand{\lp}[2]{\|#2\|_{L^{#1}(\Omega)}}
\newcommand{\io}{\int_\Omega}

\begin{document}
\footnote[0]
    {
    2010{\it Mathematics Subject Classification}\/. 
    Primary: 35B44; Secondary: 35K65, 92C17.
    }

\begin{center}
    \Large{{\bf 
    Finite-time blow-up in a quasilinear 
    degenerate chemotaxis system with flux limitation    
    }}
\end{center}
\vspace{5pt}
\begin{center}
     Yuka Chiyoda, 
    Masaaki Mizukami\footnote{Corresponding author}\footnote{Partially supported by JSPS Research 
    Fellowships for Young Scientists, No.\ 17J00101}, 
    Tomomi Yokota\footnote{Partially supported by Grant-in-Aid for
    Scientific Research (C), No.\ 16K05182.}
   \footnote[0]{
    E-mail:  
    {\tt g11914oboe@gmail.com},
    {\tt masaaki.mizukami.math@gmail.com},
    }\footnote[0]{ 
    {\tt yokota@rs.kagu.tus.ac.jp}
    }\\
    \vspace{12pt}
    Department of Mathematics, 
    Tokyo University of Science\\
    1-3, Kagurazaka, Shinjuku-ku, 
    Tokyo 162-8601, Japan\\
    \vspace{2pt}
\end{center}
\begin{center}    
    \small \today
\end{center}

\vspace{2pt}
%
%
%
\newenvironment{summary}
{\vspace{.5\baselineskip}\begin{list}{}{%
     \setlength{\baselineskip}{0.85\baselineskip}
     \setlength{\topsep}{0pt}
     \setlength{\leftmargin}{12mm}
     \setlength{\rightmargin}{12mm}
     \setlength{\listparindent}{0mm}
     \setlength{\itemindent}{\listparindent}
     \setlength{\parsep}{0pt}
     \item\relax}}{\end{list}\vspace{.5\baselineskip}}
\begin{summary}
{\footnotesize {\bf Abstract.} 
This paper deals with the quasilinear degenerate chemotaxis system 
with flux limitation 
    	\begin{align*}
    	\begin{cases}
         u_t = \nabla\cdot\left(\dfrac{u^p
         \nabla u}{\sqrt{u^2 + |\nabla u|^2}}
         \right) -\chi \nabla\cdot\left(
         \dfrac{u^q\nabla v}{\sqrt{1 + 
         |\nabla v|^2}}\right),
         &x\in \Omega,\ t>0,
         \\
         0 = \Delta v - \mu + u, 
         &x\in \Omega,\ t>0,  
         \end{cases}
        \end{align*} 
where $\Omega := B_R(0) \subset \mathbb{R}^n$ ($n \in \N$) 
is a ball with some $R>0$, and $\chi >0$, $p,q\geq1$, 
$\mu := \frac 1{|\Omega|} \io u_0$ 
and $u_0$ 
 is an initial data of an unknown function $u$. 
Bellomo--Winkler (Trans.\  Amer.\  Math.\  Soc.\  Ser.\  B;2017;4;31--67) 
established existence of an initial data such that the 
corresponding solution 
blows up in finite time when $p=q=1$. This paper gives 
existence of blow-up 
solutions under some condition for $\chi$ and $u_0$ 
when $1\leq p\leq q$.
}
\end{summary}
\vspace{10pt}
\newpage
\section{Introduction} \label{Sec1}
In this paper we consider the quasilinear 
degenerate chemotaxis system with flux limitation:
\begin{align}\label{P}     
   \begin{cases}
         u_t = \nabla\cdot\left(\dfrac{u^p
         \nabla u}{\sqrt{u^2 + |\nabla u|^2}}
         \right) -\chi \nabla\cdot\left(
         \dfrac{u^q\nabla v}{\sqrt{1 + 
         |\nabla v|^2}}\right),
         &x\in \Omega,\ t>0,
\\[2mm]
         0 = \Delta v - \mu + u, 
         &x\in \Omega,\ t>0,         
\\[2mm]
          \left(\dfrac{u^p\nabla u}{\sqrt{u^2 +
          |\nabla u|^2}} - \chi
          \dfrac{u^q\nabla v}{\sqrt{1 +
          |\nabla v|^2}}\right)\cdot\nu = 0, 
         &x \in \partial\Omega,\ t>0,
\\[5mm]
         u(x,0) = u_0(x),
         &x \in \Omega,
   \end{cases}
 \end{align}
 where $\Omega = B_R(0)\subset\Rn (n\in \N)$ is a ball 
with some $R > 0$, $\chi > 0$, $p,q\geq 1$ and the 
initial data $u_0$ is a function fulfilling that 
 \begin{align}\label{condi;ini1}
        u_0\in C^3 (\overline{\Omega})\  \mbox{is radially 
        symmetric and positive in}\  \overline{\Omega}\ \mbox{with}\ 
        \frac{\pa u_0}{\pa \nu} = 0\ \mbox{on}\ \pa\Omega
 \end{align}
 and where
 \begin{align}\label{def;mu}
        \mu := \frac{1}{|\Omega|}
        \int_\Omega{u_0(x)\,dx}.
 \end{align}
The system \eqref{P} 
represents the situation such that 
a cellular slime moves towards higher concentrations of the 
chemical substance, and the unknown function $u=u(x,t)$ 
describes the density of  cell and the unknown function 
$v=v(x,t)$ 
denotes the concentration of chemoattractant at $x\in \Omega$ 
and $t\ge 0$.   
This model is development of the chemotaxis system 
\begin{align}\label{KS}
u_t = \Delta u -\nabla\cdot (u \nabla v), \quad v_t=\Delta v-v+u; 
\end{align} 
thanks to effect of the flux limitation, the system \eqref{P} 
describes the case that cell diffusivity is suppressed; 
therefore the system \eqref{P} is innovative and important 
because it 
is considered a sensitive dynamics in aggregation phenomena. 

Before we introduce previous works about the system 
 \eqref{P}, 
we will recall known results about the chemotaxis system \eqref{KS}: 

The system \eqref{KS} is proposed by Keller--Segel \cite{KS} 
and is called a Keller--Segel system. About the Keller--Segel system 
it was known that the size of the initial data in some Lebesgue norm 
determines behaviour of solutions; 
in the case that $n=1$, Osaki--Yagi \cite{OY} obtained global existence 
and boundedness of classical solutions of \eqref{KS}; 
in the case that $n=2$, it is shown that there is a critical value 
$C>0$ ($C=8\pi$ in the radial setting and $C=4\pi$ in the other setting) 
such that, 
if $\lp{1}{u_0} < C$ then global solutions exist (\cite{NSY}), and 
if $m>C$ then there is an initial data satisfying that $\lp{1}{u_0}=m$ and 
the corresponding solution blows up in finite time (\cite{HW,MW}); 
in the case that $n\geq3$, Horstmann--Winkler \cite{HWI} 
asserted possibility of existence of unbounded solutions; 
Winkler \cite{W} showed that for all $m>0$ there exists an initial data 
such that $\lp1{u_0}=m$ 
and the corresponding solution blows up in finite time; 
also in the case that $n=3$, 
Cao \cite{C} established global existence and  boundedness 
under the condition that $\lp {\frac{n}{2}}{u_0}$ and 
$\lp n{\nabla v_0}$ are sufficiently small. 

The Keller--Segel system \eqref{KS} is now studied by 
many mathematicians intensively. 
Moreover, many variations of generalizations of 
the Keller--Segel system \eqref{KS} are 
 also sprightly studied. 
Here one of the important generalized problems is 
the degenerate chemotaxis system 
\[
u_t = \Delta u^p-\nabla\cdot (u^{q-1}\nabla v), \quad v_t=\Delta v-v+u,
\]
where $p\geq1$, $q\geq2$. This problem is one of the model which 
has a nonlinear diffusion suggested by Hillen--Painter \cite{HP}, 
and it is shown that the relation between $p$ and $q$ 
determines behaviour of solutions; 
in the case $q<p+\frac{2}{n}$, global solutions are 
obtained when $\Omega$ is a bounded 
domain (see \cite{IY}). In the case $q=p+\frac{2}{n}$, the result is divided 
by the size of the initial data with some critical mass $m_c=m_c(n)$. 
When 
 $\lp1{u_0}=m<m_c$ 
and $q=2$, under the 
Dirichlet--Neumann boundary condition, 
Mimura \cite{M} 
showed that there are global 
solutions when $\Omega$ is a bounded domain. 
On the other hand, if $m>m_c$, 
then Lauren\c{c}ot--Mizoguchi \cite{LM} established that 
existence of an initial data such that $\lp1{u_0}=m$ 
and the corresponding 
solution blows 
up in finite time when $q=2$, $\Omega=\Rn$ and $n=3,4$. 
In the case $q>p+\frac{2}{n}$, it is known that there exists an initial 
data such that 
the corresponding solution blows up in finite time 
(see \cite{HIY}). 
Thus, it is clear that the relation between $p,q$ and $n$ 
strongly affects behaviour of solutions. 

On the other hand, 
the system which describes the situation such that 
the movement of the species is suppressed, that is, 
the chemotaxis system with flux limitation 
\begin{align}\label{P0}
    	\begin{cases}
         u_t = \nabla\cdot\left(D_u(u,v)\dfrac{u
         \nabla u}{\sqrt{u^2 + |\nabla u|^2}}
         \right) - \nabla\cdot\left(S(u,v)
         \dfrac{u\nabla v}{\sqrt{1 + 
         |\nabla v|^2}}\right)+H_1(u,v),
     & 
         \\[5mm]
         v_t = D_v\Delta v + H_2(u,v),
       \qquad  x\in \Omega,\ t>0   &
         \end{cases}
        \end{align}
is proposed by Bellemo--Winkler \cite{BW0}, 
where $D_u$ and $D_v$ show 
properties of 
diffusion of the species and the chemoatractant, respectively, 
and $S$ represents the chemotactic interaction, 
and $H_1$ and $H_2$ are mechanisms of propagation, 
degeneration, and interaction. 
Since it has not been known whether there exist valid 
functions 
like an energy function and a Lyapunov function yet, 
the system \eqref{P0} seems to be difficult. 
Therefore, Bellomo--Winkler 
\cite{BW0,BW} have considered the following 
 simplified system 
\begin{align}\label{P1}     
         u_t = \nabla\cdot\left(\dfrac{u
         \nabla u}{\sqrt{u^2 + |\nabla u|^2}}
         \right) -\chi \nabla\cdot\left(
         \dfrac{u\nabla v}{\sqrt{1 + 
         |\nabla v|^2}}\right),\quad  
        0 = \Delta v - \mu + u.         
 \end{align}
In this system, Bellomo--Winkler \cite{BW0} overcame the 
difficulty, and they showed   
global existence when $\chi<1$. 
On the other hand,  if $\chi>1$ and
\begin{align}\label{condi;m}
  \begin{cases}
    m> \dfrac{1}{\sqrt{\chi^2 -1}} & \mbox{if} \ n=1, 
\\[5mm]  
    m> 0 \ \mbox{is arbitrary} & \mbox{if} \ n\ge 2, 
  \end{cases}
\end{align}
Bellomo--Winkler \cite{BW} found an initial data such that 
the corresponding solution of \eqref{P1} blows up in finite time.
%
However, the problem \eqref{P0} has not been 
studied yet when $D_u$ and $S$ are general; 
in view of the study of the Keller--Segel system, 
the case that $D_u(u,v)=u^{p-1}$ and $S(u,v)=u^{q-1}$ 
in 
\eqref{P0}, i.e., the system \eqref{P} seems to 
be 
one of important problems. 
 Recently global existence of solutions to system \eqref{P} 
was shown  
when $p > q+1-\frac{1}{n}$ (see \cite{OMY}). 
From the results in the degenerate chemotaxis system 
we can expect that some largeness condition for $q$ 
derives existence of blow-up solutions. 

The purpose of this paper is to determine 
the condition for $p$ and $q$ such that the corresponding 
solution blows up in finite time. 
Here we need to establish 
different methods because 
we cannot 
adopt the same argument as 
 in \cite{BW} when $p<q$ holds.

Now main 
results read as follows.

%
%
%
%
\begin{thm}\label{mainthm1}
Let $\Omega := B_R(0)\subset\Rn$ $(n\in \N)$ with $R > 0$ 
and suppose that  $1\leq p \le  q$. 
\begin{item}
{\rm (i)} If $n=1$, then for all $\chi > 1$ $(\chi >0$ when $q > p)$, 
there exists $m_c=m_c(\chi,p,q,R) > 0$ with the 
following property\/{\rm :} If 
 \begin{align}\label{condi;mass}
    m > m_c,
 \end{align}
then there exists a nondecreasing function $M_m \in  C^0([0,R])$ 
satisfying  $\sup_{r\in(0,R)}\frac{M_m(r)}{|B_r(0)|}< \infty$, 

$M_m(R)\leq m$, and that for all $u_0 \in C^3(\overline{\Omega})$ 
with $\int_{\Omega}{u_0(x)\,dx} = m$ and
 \begin{align}\label{condi;ini3}
     \int_{B_r(0)}{u_0(x)\,dx}\geq M_m(r)
 \end{align}
for all $r\in[0,R]$, there exists $T^\ast\in(0,\infty)$ 
such that a corresponding solution 
$(u,v)$ of \eqref{P} blows up in 
finite time $T^\ast$ in the sense that
 \begin{align}\label{blow-up}
     \limsup_{t\nearrow T^\ast}\lp{\infty}{u(\cdot,t)}
      = \infty.
 \end{align} 
\end{item}
\begin{item}
{\rm (ii)} If $n \geq 2$ and $m > 0$, then 
for all $\chi >0$ satisfying 
 \begin{align}\label{condi;chi}
     \chi > \left(\frac{mn}{\omega_n R^n}\right)^{p-q}, 
 \end{align}
where $\omega_n$ defines the $(n-1)$-dimensional 
measure of the unit sphere in $\Rn$, there exists a nondecreasing 
function $M_m \in  C^0([0,R])$ satisfying  $\sup_{r\in(0,R)}\frac{M_m(r)}{|B_r(0)|}< \infty$, 
 and $M_m(R)\leq m$, such that for all 
$u_0 \in C^3(\overline{\Omega})$ with 
$\int_{\Omega}{u_0(x)\,dx} = m$ and 
\eqref{condi;ini3} 
 for all $r\in[0,R]$, there exists $T^\ast\in(0,\infty)$ such that 
 the 
  corresponding solution 
$(u,v)$ of \eqref{P} blows up in 
finite time $T^\ast$ in the sense of \eqref{blow-up}.  
\end{item}
\end{thm}
%
%
%
%
\begin{remark}\label{Re1.5}
This theorem shows existence of blow-up solutions to \eqref{P}  
when $q\ge p$. 
On the other hand, 
in \cite{OMY} 
global existence of solutions is shown 
when $p> q+1-\frac{1}{n}$. 
However, except for the case $n=1$, 
behaviour of solutions in the case $q < p \le q+1-\frac 1n$ 
is still an open problem (see Figure \ref{Graph}).
\begin{figure}[h]
\begin{center}
\input{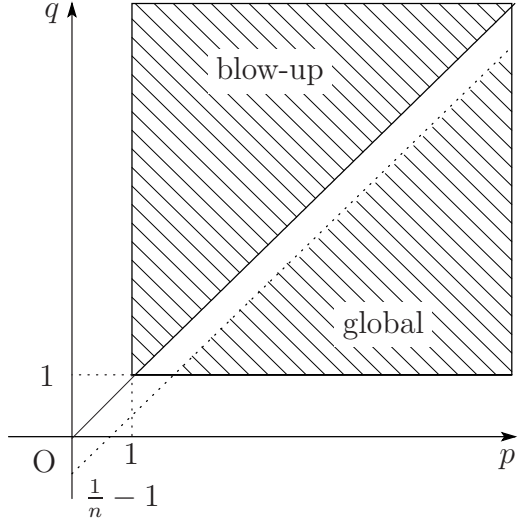}
\end{center}
\caption{Classification of behaviour}
\label{Graph}
\end{figure}
\end{remark}
%
%
%
%
\begin{remark}\label{Re1}
If $p=q=1$, then the 
 condition for $m$ 
connects with that in \cite{BW}; 
indeed, the constant $m_c(\chi,p,q,R)$ 
in \eqref{condi;mass} is given by 
 \[
    m_c=
    \inf \left\{m\  \Bigg|\  \exists\, \lambda \in 
    \left(\frac{5-\sqrt{17}}{2},1\right];\,\frac{\left\{1-\frac{(1-\lambda)^2}
    {3\lambda-1}\right\}m \chi}
    {\sqrt{\frac{1}{\lambda^2}+
    \frac{(1-\lambda)^2}{\lambda^2(3\lambda-1)}
    + m^2}}=\left(\frac{m}{\omega_n}\cdot 
    \frac{(1-\lambda)^2}{\lambda(\lambda+1)R}\right)^{p-q} \right\} 
 \]
(see \eqref{def;ab} and \eqref{condi;mc}). 
Thus, in 
 the case $p=q$, since 
 \begin{align*}
  m_c(\chi,p,p,R)
 & 
  =\inf \left\{m\  \Bigg|\  \exists\, \lambda \in 
  \left(\frac{5-\sqrt{17}}{2},1\right];\,\frac{\left\{1-\frac{(1-\lambda)^2}
  {3\lambda-1}\right\}m \chi}
 {\sqrt{\frac{1}{\lambda^2}+\frac{(1-\lambda)^2}{\lambda^2(3\lambda-1)}
 + m^2}} = 1 \right\}
\\
    & =\inf
    \left\{m\  \Bigg|\  \frac{m \chi}
 {\sqrt{1 + m^2}} = 1 \right\}
    =\frac{1}{\sqrt{\chi^2 -1}}
 \end{align*}
and moreover 
 \[
    \left(\frac{mn}{\omega_n R^n}\right)^{p-p}=1,
 \]
the conditions \eqref{condi;mass} 
and \eqref{condi;chi} are reduced to \eqref{condi;m}. 
Thus Theorem \ref{mainthm1} is a 
generalization of the previous work \cite{BW}.
\end{remark} 

In view of Remark \ref{Re1} we have the following corollary.
%
%
%
%
\begin{corollary}
Let $\Omega := B_R(0)\subset\Rn$ $(n\in \N)$ with 
$R > 0$ and let $p=q\ge 1$, and suppose that $\chi>1$ 
and \eqref{condi;m}. 
Then there exists a nondecreasing function 
$M_m\in C^0([0,R])$ fulfilling 
$\sup_{r\in (0,R)} \frac{M_m(R)}{|B_r(0)|} < \infty$ 
and $M_m(R) \le m$, 
which is such that whenever $u_0$ satisfies \eqref{condi;ini1}, 
as well as \eqref{condi;ini3}  for all $r\in[0,R]$, 
the solution $(u,v)$ of \eqref{P} blows up in finite time 
in the sense of \eqref{blow-up}. 
\end{corollary}

Theorem \ref{mainthm1} gives the following byproduct; 
since arguments similar to those in the proof of \cite[Proposition 1.2]{BW} 
enable us to see this proposition, we only write the statement. 
%
%
%
%
\begin{prop}\label{prop}
Let $n\in \N$, $R>0$, $\Omega := B_R(0)\subset\Rn$ 
 and $\chi > 1$.
 \begin{item}
 {\rm(i)} Let $m > 0$ satisfy \eqref{condi;mass}. 
 Then there exists a radially symmetric positive 
 $u_m \in C^{\infty}(\overline{\Omega})$ 
 which is such that 
 \[
 \frac{\partial u_m}{\partial \nu} 
 = 0\ \mbox{on}\ \partial\Omega\quad \mbox{and}\ \int_\Omega{u_m}=m, 
 \]
 and for which it is possible to choose 
 $\varepsilon > 0$ with the property that whenever 
 $u_0$ 
 satisfies \eqref{condi;ini1} 
  as well as
 \[
\lp{\infty}{u_0-u_m}\leq \varepsilon,
 \]   
  the corresponding solution of \eqref{P} 
  blows up in finite time. 
\end{item}
\begin{item}
  {\rm(ii)} Given any $u_0$ fulfilling \eqref{condi;ini1}, 
  one can find functions 
   $u_{0k}$, 
   $k\in \N$, which
   satisfy \eqref{condi;ini1} and 
   \[
u_{0k}
\to u_0 
   \quad \mbox{in} \ L^p(\Omega)\quad \mbox{as}\ k\to\infty   
   \]
   for all $p\in(0,1)$,
    and which are such that for all $k\in\N$ 
    the solution of \eqref{P} emanating from 
    $u_{0k}$
    blows up in finite time.
 \end{item}
\end{prop}

The proof of the main result 
 is based on that of \cite[Theorem1.1]{BW}. 
Thus in the same way, we introduce $s:=r^n$ for $r\in [0,R]$ 
and the mass accumulation function $w=w(s,t)$ defined as 
  \[
      w(s,t):= \int_0^{s^{\frac{1}{n}}}{r^{n-1}u(r,t)\,dr},
  \]
and  
then, a combination of the fact $u(r_\ast,t) \geq \frac{w(s_\ast,t)}{s_\ast}$ 
given by using $u \geq w_s$ and the mean value theorem and 
that $w$ is the solution of a scalar parabolic  equation yields 
that the core of the proof is to find a suitable 
subsolution 
$\underline{w}$ 
 such that for some $T>0$ and some $s_\ast\in(0,R^n)$, 
  \begin{align*}
      \frac{\underline{w}(s_\ast,t)}{s_\ast}\to \infty\  \mbox{as}\  t\nearrow T.
  \end{align*}
In the previous study \cite{BW}, the interval $(0,R^n)$  
is divided three parts, 
and in very inner region thanks to construction of the subsolution 
$\underline{w}$ using the structure of a quadratic function, 
we obtain  a 
suitable estimate. However, in this paper 
we cannot 
establish it from 
the same argument when $p<q$. 
Therefore, 
adopting a new subsolution $\underline{w}$ consisted by 
 an exponential function, we can prove 
existence of an initial data such that the corresponding 
solution blows up in finite time. 
In this proof, the key idea is to employ a new 
viewpoint in the proof of some suitable estimate; by 
establishing a new estimate 
where the effect of the aggregation come from 
chemotactic interaction 
works adequately 
(see Lemma \ref{tool;very inner}), 
we can attain a useful estimate. 

This paper is organized as follows. In Section 2  
we recall local existence in \eqref{P} and we 
consider the mass accumulation function and 
a scalar parabolic equation. 
In order to use the comparison argument as in 
\cite[Lemma 5.1]{BW} we construct subsolutions 
 and confirm properties in Section 3. Finally, we prove 
 existence of blow-up solutions in Section 4. 
\section{Local existence and a parabolic problem
 satisfied by the mass accumulation}\label{Sec2}
In this section we provide a local existence result and 
 a mass accumulation satisfies some parabolic problem. 
First, we  recall a local existence result; 
the following result was shown in \cite[Theorem 1.1]{OMY}.
%
%
%
%
\begin{lem}\label{localexistence}
Suppose that $u_0$ complies with \eqref{condi;ini1}. 
Then there exist $\tmax\in(0,\infty]$ and 
a pair $(u,v)$ 
of positive radially symmetric 
functions $u\in C^{2,1}(\overline{\Omega}\times[0,\tmax))$ 
and $v\in C^{2,0}(\overline{\Omega}\times[0,\tmax))$ which solve 
\eqref{P} classically in $\Omega\times(0,\tmax)$, and which are 
such that 
\begin{align}\label{local}
\mbox{if}\quad\tmax<\infty,\quad\mbox{then}\quad
\limsup_{t\nearrow\tmax}\|u(\cdot,t)\|_{L^{\infty}(\Omega)}
=\infty.
\end{align}
\end{lem}
In the following let
$\Omega := B_R(0)\subset \Rn$ ($n \in \N$) with some $R > 0$ 
and let $u_0$ satisfy \eqref{condi;ini1}, and denote by 
$(u,v)=(u(r,t),v(r,t))$  the  radially symmetric local solution of 
\eqref{P} and by $\tmax$ its maximal existence time 
obtained in Lemma \ref{localexistence}. 
Moreover, we introduce the mass accumulation function 
$w$ and the parabolic operator $\mathcal{P}$ such that
 \begin{align}\label{def;w}
     w(s,t) := \int_0^{s^{\frac{1}{n}}}{r^{n-1}u(r,t)\,dr}
 \end{align}
for $s\in[0,R^n]$ and $t\in[0,T)$, and
 \begin{align}\label{def;P}
     (\mathcal{P}\widetilde{w})(s,t) := \widetilde{w}_t-n^{p+1}
     \cdot\frac{s^{2-\frac{2}{n}}\widetilde{w}_s^p
     \widetilde{w}_{ss}}
     {\sqrt{\widetilde{w}_s^2+n^2s^{2-\frac{2}{n}}
     \widetilde{w}_{ss}^2}}-n^q\chi\cdot
     \frac{(\widetilde{w}-\frac{\mu}{n}s)\widetilde{w}_s^q}
     {\sqrt{1+s^{\frac{2}{n}-2}
     (\widetilde{w}-\frac{\mu}{n}s)^2}}.
 \end{align}
 
If $\widetilde{w}\in C^1((0,R^n)\times(0,T))$ is such that 
$\widetilde{w}_s>0$ and $\widetilde{w}(\cdot,t)\in W^{2,\infty}(0,R^n)$ 
for all $t\in(0,T)$, then the expression $\mathcal{P}\widetilde{w}$ 
is well-defined. Now we show that the function $w$ defined as 
\eqref{def;w} fulfills these corresponding condition. 
Thus, the following lemma yields that the function $w$ 
satisfies some parabolic problem (see \cite[Lemma 2.1]{BW}).
%
%
%
%
\begin{lem}\label{baselem}
Let $n\in\N$, $\chi > 0$. Then 
for $T > 0$ and some nonnegative 
radially symmetric $u_0\in C^0(\overline{\Omega})$, whenever $(u,v)$ 
is a positive radially symmetric classical 
solution of \eqref{P} in $\Omega \times [0,T)$, the function $w$ 
defined as \eqref{def;w} satisfies
 \begin{align}\label{lem;base}     
     \begin{cases}
      (\mathcal{P}w)(s,t) = 0, &s\in(0,R^n),\ t\in(0,T),
      \\
      w(0,t) = 0,\quad w(R^n,t) = \dfrac{m}{\omega_n},
      &t\in(0,T),
      \\
      w(s,0) = \int_0^{s^\frac{1}{n}}{r^{n-1}u_0(r)\,dr},
      &s\in[0,R^n],
     \end{cases}
 \end{align}
where $m := \int_\Omega{u_0(x)\,dx}$ and 
where $\omega_n$ denotes the $(n-1)$-dimensional 
measure of the unit sphere in $\Rn$.
\end{lem}
\begin{proof}
An argument 
 similar to that in 
 the proof of \cite[Lemma 2.1]{BW}
 implies the conclusion of this lemma.
\end{proof}
\section{Construction of subsolutions for \eqref{lem;base}}\label{Sec3}

In this section we construct a subsolution 
$\underline{w}$ of \eqref{lem;base}. 
Then, using a suitable comparison 
principle (see \cite[Lemma 5.1]{BW}) to obtain $w\geq \underline{w}$ 
in $[0,R^n]\times[0,T)$, we derive that $u(r,t)$ blows up in 
finite time $T>0$ (fixed later). In Section \ref{subsolution} we prepare 
a family of functions, and define $\underline{w}$. In Sections \ref{outer}, 
\ref{inner}, \ref{veryinner} and \ref{intermadiate} 
we divide $[0,R^n]$ into three parts 
and show properties of a subsolution 
 $\underline{w}$ in respective regions. 
%
%
%
%
\subsection{Constructing a family of 
candidates.}\label{subsolution}
In order to construct subsolutions $\underline{w}$ for \eqref{lem;base} 
we first provide some parameter and some function; 
for $\lambda\in [\frac{1}{3},1]$ we put
 \begin{align}\label{def;ab}
     a_\lambda := \frac{(1-\lambda)^2}{2\lambda}\geq 0\quad
     \mbox{and}\quad b_\lambda := \frac{3\lambda-1}
     {2\lambda}\geq 0
 \end{align}
and define 
\begin{align}\label{def;phi}
         \varphi(\xi) :=
           \begin{cases}
              \dfrac{2\lambda}{de^d}(e^{d\xi}-1)
              &\mbox{if}\ \xi\in [0,1),
              \\[4mm]
              1-\dfrac{a_\lambda}{\xi-b_\lambda}
              &\mbox{if}\ \xi \geq 1,
           \end{cases}
 \end{align}
where $1<d<2$ is such that
\begin{align}\label{def;d}
        (2-d)e^d-2=0.
\end{align}  
Here we note that 
there is a solution $d\in (1,2)$ of \eqref{def;d}; 
indeed, since 
\[
  ((2-d)e^d-2) |_{d=1} = e-2 > 0 
\quad 
  \mbox{and} 
\quad 
  ((2-d)e^d-2) |_{d=2} = -2<0 
\]
hold, the intermediate value theorem enables us to 
find $d\in (1,2)$ satisfying \eqref{def;d}. 
  Then we can show that
 $\varphi \in C^1([0,\infty))\cap W^{2,\infty}(0,\infty)\cap 
  C^2([0,\infty)\setminus\{1\})$ with 
  \begin{align}\label{def;phi'}
        \varphi'(\xi) =
           \begin{cases}
              2\lambda e^{d(\xi-1)}
              &\mbox{if}\ \xi\in [0,1),
              \\[4mm]
              \dfrac{a_\lambda}{(\xi-b_\lambda)^2}
              &\mbox{if}\ \xi \geq 1
           \end{cases}
 \end{align}
 and
 \begin{align}\label{def;phi''}
         \varphi''(\xi) =
           \begin{cases}
              2d\lambda e^{d(\xi-1)}
              &\mbox{if}\ \xi\in [0,1),
              \\[4mm]
              -\dfrac{2a_\lambda}{(\xi-b_\lambda)^3}
              &\mbox{if}\ \xi \geq 1.
           \end{cases}
 \end{align}
In particular, $\varphi'(\xi) > 0$ for all $\xi \geq 0$. 
Then we have to choose 
$\lambda \in (\frac{5-\sqrt{17}}{2},1]$ suitably in the case that 
$n=1$ 
(see Lemma \ref{estimate1;intermediate}), and we 
must fix $\lambda=\frac{1}{3}$ in the case $n\geq2$ 
(see Lemma \ref{estimate2;intermediate}). 
The following lemma has already been proved
 in the proof of \cite[Lemma 3.1]{BW}. Thus we 
 recall only 
  the statement of the lemma.
%
%
 
\begin{lem}\label{def;A,D,E,N}
 Let $n \in \N$, $m > 0$, $\lambda\in[\frac{1}{3},1]$, $K > 0$, 
 $T >0$, and suppose that $B \in C^1([0,T))$ satisfies that
 $B(t)\in(0,1)$, $K\sqrt{B(t)} < R^n$ for all 
 $t\in[0,T)$ and 
 \begin{align}\label{condi;B1}
     B(t)\leq \frac{K^2}{4(a_\lambda+b_\lambda)^2}   
 \end{align}
for all $t\in[0,T)$, where $a_\lambda$ and $b_\lambda$ are given by 
\eqref{def;ab}. Then the following 
functions are well-defined\/{\rm :}
 \begin{align}\label{def;A}
    A(t) := \frac{m}{\omega_n}\cdot\frac
    {K^2-2b_\lambda K\sqrt{B(t)}+b^2_\lambda B(t)}
    {N(t)} 
    \quad \mbox{for} \ t\in[0,T)
 \end{align}
and
 \begin{align}\label{def;D}
    D(t) := \frac{m}{\omega_n}\cdot\frac{a_\lambda}
    {N(t)}     \quad \mbox{for} \ t\in[0,T)
 \end{align}
as well as
 \begin{align}\label{def;E}
    E(t) := \frac{m}{\omega_n}-R^nD(t)=
    \frac{m}{\omega_n}\cdot\frac
    {K^2-(a_\lambda+b_\lambda)(2K\sqrt{B(t)}-
    b_\lambda B(t))}{N(t)} 
        \quad \mbox{for} \ t\in[0,T)
 \end{align}
with
 \begin{align}\label{def;N}
    N(t) := K^2+a_\lambda R^n-(a_\lambda +b_\lambda)
    (2K\sqrt{B(t)}-b_\lambda B(t))
     \quad \mbox{for} \ t\in[0,T).  
 \end{align}
Furthermore, we have
 \[
    A'(t) = \frac{m}{\omega_n}\cdot\frac{\Bigl(\frac{K}
    {\sqrt{B(t)}}-b_\lambda\Bigr)\cdot(a_\lambda K^2-
    a_\lambda b_\lambda R^n)\cdot B'(t)}{N^2(t)}
 \]
as well as
 \begin{align}\label{def;D'}
    D'(t) = \frac{m}{\omega_n}\cdot\frac
    {a_\lambda(a_\lambda+b_\lambda)\cdot
    \Bigl(\frac{K}
    {\sqrt{B(t)}}-b_\lambda\Bigr)\cdot B'(t)}{N^2(t)}
 \end{align}
for all $t\in(0,T)$.
\end{lem}
 
%
%

Using these definitions, 
we can express clearly our comparison function 
$\underline{w}$. 
Letting $K>0$ be a constant fixed later 
and letting $B$ be a function chosen suitably later, 
we will give a composite structure of $\underline{w}$ by separating 
$[0,R^n]$ into two parts 
(an inner region and an outer region). 
 
\begin{lem}\label{def;subsolution}
 Let $n \in \N$, $m > 0$, $\lambda\in[\frac{1}{3},1]$, 
 $K > 0$ and $T>0$, and let $B \in C^1([0,T))$ satisfy 
 that $B(t)\in(0,1)$, $K\sqrt{B(t)} < R^n$ 
 and \eqref{condi;B1} hold
 for all $t\in[0,T)$. Suppose that
 \[
         \underline{w}(s,t) :=
           \begin{cases}
              w_{\rm in}(s,t)
              &\mbox{if}\ t\in[0,T)\ \mbox{and}\ 
              s\in \big[0,K\sqrt{B(t)}\big],
              \\[1mm] 
              w_{\rm out}(s,t)
              &\mbox{if}\ t\in[0,T)\ \mbox{and}\ 
              s\in \big(K\sqrt{B(t)},R^n\big],
           \end{cases}
 \]
where
 \begin{align}\label{def;win}
         w_{\rm in}(s,t) := A(t)\varphi(\xi),\quad\xi = \xi(s,t) := \frac{s}{B(t)}
 \end{align}
for $t \in[0,T)$, $s\in[0,K\sqrt{B(t)}]$ with $\varphi$ and $A$ introduced 
as \eqref{def;phi} and \eqref{def;A}, respectively, and where
 \begin{align}\label{def;wout}
         w_{\rm out}(s,t) := D(t)s+E(t)
 \end{align}
for $t\in[0,T)$ and $s\in[K\sqrt{B(t)},R^n]$ 
with $D$ and $E$ as in \eqref{def;D} 
and \eqref{def;E}, respectively. 
Then $\underline{w}$ is 
well-defined and satisfies 
\[
  \underline{w}\in C^1([0,R^n]\times[0,T))  
\] 
and $\underline{w}(\cdot,t)\in W^{2,\infty}(0,R^n)\cap 
C^2([0,R^n]\setminus\{B(t),K\sqrt{B(t)}\})$ 
for all $t\in[0,T)$ as well as
 \[
         \underline{w}(0,t) = 0\quad \mbox{and}\quad
          \underline{w}(R^n,t) = \frac{m}{\omega_n}
 \]
for all $t\in(0,T)$.
\end{lem}
\begin{proof}
An argument similar to that in the proof of \cite[Lemma 3.2]{BW}
 implies the conclusion of this lemma.
\end{proof}
 
%
%
\subsection{Subsolution properties: Outer region.}\label{outer}

First, we will consider in the outer region. In the following 
lemma, we show that if the function $B$ constructing $\underline{w}$ 
is suitably small and fulfilling a differential inequality, 
then $\underline{w}$ is a subsolution of 
\eqref{lem;base} in the region.

%
%

\begin{lem}\label{estimate;outer}
 Let $n \in \N$, $\chi > 0$, $m > 0$, $\lambda\in[\frac{1}{3},1]$, 
 $K > 0$, $T>0$ and $B_0\in(0,1)$ fulfill  
 $K\sqrt{B_0} < R^n$ and 
  \[\label{condi;B2}
     B_0\leq \frac{K^2}{16(a_\lambda+b_\lambda)^2}  
  \]
with $a_\lambda$ and $b_\lambda$ taken from 
\eqref{def;ab}. Then if $B\in C^1([0,T))$ 
is positive and nonincreasing and satisfies that
 \begin{align}\label{condi;B3}
    	\begin{cases}
          B'(t)\geq -\dfrac{a_\lambda^{q-1}(nm)^q\chi K}
          {2(a_\lambda+b_\lambda)
          (K^2+a_\lambda R^n)^{q-1}\omega_n^qR^n\sqrt{1+K^{\frac{2}{n}-2}
          \cdot \frac{m^2}{\omega_n^2}}}\cdot B^{1-\frac{1}{2n}}(t),
         \\[4mm]
          B(0)\leq B_0
         \end{cases}
 \end{align}
 for all $t\in(0,T)$, 
 then 
 the function $w_{\rm out}$ given by \eqref{def;wout} fulfills that
 \begin{align}\label{ineq;Pwout}
     (\mathcal{P}w_{\rm out})(s,t) \leq 0
 \end{align}
for all t  $\in(0,T)$ and all $s\in(K\sqrt{B(t)},R^n)$
 with $\mathcal{P}$ defined in \eqref{def;P}.
\end{lem}
\begin{proof}
The proof is based on that of \cite[Lemma 3.3]{BW}.
Recalling that $E(t) = \frac{m}{\omega_n}-R^nD(t)$ 
for all $t\in(0,T)$ by \eqref{def;E}, we have  
 \begin{align}\label{equ;wout}
     w_{\rm out}(s,t) = D(t)s+E(t) 
     = \frac{m}{\omega_n}-D(t)\cdot (R^n-s)
 \end{align}
for all $t\in(0,T)$ and all $s\in(K\sqrt{B(t)},R^n)$. 
Straightforward calculations together with \eqref{def;D'} yield that
 \begin{align}\label{equ;wout_t}
     (w_{\rm out})_t(s,t) &= -D'(t)\cdot (R^n-s)
     \\
      &= -\frac{m}{\omega_n}\cdot\frac{a_\lambda
      (a_\lambda+b_\lambda)\Bigl(\frac{K}{\sqrt{B(t)}}
      -b_\lambda\Bigr)}{N^2(t)}\cdot B'(t)\cdot (R^n-s)\nonumber
 \end{align}
for all $t\in(0,T)$ and all $s\in(K\sqrt{B(t)},R^n)$ 
with $N$ as in \eqref{def;N}. 
In order to obtain an estimate for $\mathcal{P} w_{\rm out}$, 
noting from \eqref{equ;wout_t} and the fact $(w_{\rm out})_{ss}\equiv 0$  
that 
\begin{align}\label{iden;Pwout}
(\mathcal{P} w_{\rm out})(s,t) &= (w_{\rm out})_t(s,t)-n^{p+1}
     \cdot\frac{s^{2-\frac{2}{n}}\{(w_{\rm out})_s\}^p
     (w_{\rm out})_{ss}}
     {\sqrt{{(w_{\rm out})_s}^2+n^2s^{2-\frac{2}{n}}
     {(w_{\rm out})_{ss}}^2}}+I(s,t) 
    \\ \notag
    &= -\frac{m}{\omega_n}\cdot\frac{a_\lambda
      (a_\lambda+b_\lambda)\Bigl(\frac{K}{\sqrt{B(t)}}
      -b_\lambda\Bigr)}{N^2(t)}\cdot B'(t)\cdot (R^n-s) 
      +I(s,t),
\end{align}
where 
\begin{align}\label{def;I}
     I(s,t) := -
     n^q\chi\cdot\frac{(w_{\rm out}-
     \frac{\mu}{n}s)\cdot\{(w_{\rm out})_s\}^q}
     {\sqrt{1+s^{\frac{2}{n}-2}(w_{\rm out}-
     \frac{\mu}{n}s)^2}}  
\end{align}
for $t\in(0,T)$ and $s\in(K\sqrt{B(t)},R^n)$, we shall derive an estimate for $I(s,t)$. 
Since \eqref{equ;wout} holds, we first obtain from the identity 
$\frac{\mu}{n}=\frac{m}{\omega_n R^n}$ that 
 \begin{align}\label{equ;wout2}
    w_{\rm out}-\frac{\mu}{n}s &= 
     \frac{m}{\omega_n}-
     D(t)\cdot(R^n-s)-\frac{m}{\omega_n R^n}\cdot s
     \\ \nonumber
     &= \left(\frac{m}{\omega_n R^n}-D(t)\right)\cdot(R^n-s) 
 \end{align}
for all $t\in(0,T)$ and all $s\in(K\sqrt{B(t)},R^n)$. 
Noticing from 
arguments similar to those in the proof of \cite[Lemma 3.3]{BW} that 
 \[
    \sqrt{1+s^{\frac{2}{n}-2}\left(w_{\rm out}(s,t)
    -\frac{\mu}{n}s\right)^2}
    \leq
    \sqrt{1+K^{\frac{2}{n}-2}
    \cdot\frac{m^2}{\omega_n^2}}\cdot B^{\frac{1}{2n}
    -\frac{1}{2}}(t)
 \]
for all $t\in(0,T)$ and all $s\in(K\sqrt{B(t)},R^n)$,   
we infer from \eqref{def;I} and \eqref{equ;wout2} that 
 \begin{align}\label{inequ;I1}
     -I(s,t)\geq n^q\chi\frac{(\frac{m}{\omega_n R^n}
     -D(t))\cdot D^q(t)}{\sqrt{1+K^{\frac{2}{n}-2}
    \cdot\frac{m^2}{\omega_n^2}}}\cdot B^{\frac{1}{2}
    -\frac{1}{2n}}(t)\cdot (R^n-s)
 \end{align}
holds for all $t\in(0,T)$ and all $s\in(K\sqrt{B(t)},R^n)$. 
Here in light of the definitions of $D$ and $N$ 
(see \eqref{def;D} and \eqref{def;N}, respectively), we can see that  
 \begin{align}\label{equ3-1}
     \left(\frac{m}{\omega_n R^n}
     -D(t)\right)\cdot D^q(t) &= \left(\frac{m}{\omega_n R^n}
     -\frac{m}{\omega_n}\cdot\frac{a_\lambda}{N(t)}\right)
     \left(\frac{m}{\omega_n}\cdot\frac{a_\lambda}{N(t)}
     \right)^q
     \\
     &= \frac{m^{q+1}a_\lambda^q}{\omega_n^{q+1}
     N^q(t)}\cdot\frac{N(t)-a_\lambda R^n}{R^n N(t)}
     \notag
     \\
     &= \frac{m^{q+1}a_\lambda^q}{\omega_n^{q+1}
     N^2(t)}\cdot\frac{K^2-2(a_\lambda +b_\lambda)
     K\sqrt{B(t)}+(a_\lambda +b_\lambda)b_\lambda B(t)}
     {R^n {N(t)}^{q-1}}\notag
 \end{align}
for all $t\in(0,T)$ and all $s\in(K\sqrt{B(t)},R^n)$.  
Moreover, the fact 
$B_0<\frac{4K^2}{b_\lambda^2}$ by \eqref{condi;B2} leads to that 
 \begin{align}\label{inequ3-3}
     N(t) = K^2+a_\lambda R^n-(a_\lambda +b_\lambda)
    (2K\sqrt{B(t)}-b_\lambda B(t))\leq K^2+a_\lambda R^n
    \\ \nonumber
 \end{align}
for all $t\in(0,T)$, 
a combination of the relation 
\[
    K^2-2(a_\lambda+b_\lambda)K\sqrt{B(t)}
    \geq\frac{1}{2}K^2
\]
(by \eqref{condi;B1}),  
 \eqref{equ3-1}, \eqref{inequ3-3} and the fact $(a_\lambda+b_\lambda)
 b_\lambda B(t)\geq 0$ yields that
 \begin{align*}
     \left(\frac{m}{\omega_n R^n}
     -D(t)\right)\cdot D^q(t)&\geq  \frac{m^{q+1}a_\lambda^q}
     {\omega_n^{q+1}N^2(t)}\cdot\frac{K^2-2(a_\lambda 
     +b_\lambda )K\sqrt{B(t)}}{R^n 
     (K^2+a_\lambda R^n)^{q-1}}\notag
     \\
     &\geq  \frac{m^{q+1}a_\lambda^q}
     {\omega_n^{q+1}N^2(t)}\cdot\frac{K^2}{2R^n 
     (K^2+a_\lambda R^n)^{q-1}}\notag
     \\
     &=\frac{m}{\omega_n}\cdot\frac{a_\lambda}{N^2(t)}
     \cdot\frac{m^q a_\lambda^{q-1}}{\omega_n^q}\cdot
     \frac{K^2}{2R^n (K^2+a_\lambda R^n)^{q-1}}
 \end{align*}
for all $t\in(0,T)$.  
Therefore we verify from \eqref{inequ;I1} that 
 \begin{align}\label{inequ;I2}
     -I(s,t)\geq n^q \chi \cdot \frac{\frac{m}{\omega_n}
     \cdot\frac{a_\lambda}{N^2(t)}\cdot\frac
     {m^q a_\lambda^{q-1}}{\omega_n^q}\cdot
     \frac{K^2}{2R^n (K^2+a_\lambda R^n)^{q-1}}}
     {\sqrt{1+K^{\frac{2}{n}-2}\cdot\frac{m^2}
     {\omega_n^2}}}\cdot B^{\frac{1}{2}-\frac{1}{2n}}(t)
     \cdot(R^n-s)
 \end{align}
for all $t\in(0,T)$ and all $s\in(K\sqrt{B(t)},R^n)$.
Here putting 
 \[
 c_1 := \frac{a_\lambda^{q-1}(nm)^q\chi}
          {
          2(K^2+a_\lambda R^n)^{q-1}\omega_n^qR^n\sqrt{1+
          K^{\frac{2}{n}-2}
          \cdot \frac{m^2}{\omega_n^2}}} 
 \]
and using the fact 
$(a_\lambda +b_\lambda)b_\lambda B'(t) \leq 0$ for all $t\in(0,T)$, 
from \eqref{iden;Pwout} and \eqref{inequ;I2} we can confirm that 
 \begin{align}\label{inequ;Pwout}
      (\mathcal{P}w_{\rm out})(s,t)&\leq 
      \frac{m}{\omega_n}\cdot\frac{a_\lambda}{N^2(t)}
      \cdot(R^n-s)\cdot\Biggl\{\left(-\frac{(a_\lambda +b_\lambda)K}
      {\sqrt{B(t)}}+(a_\lambda +b_\lambda)b_\lambda\right)
      \cdot B'(t)
      \\
      &\quad\,-n^q\chi \cdot \frac{\frac{m^q a_\lambda^{q-1}}
      {\omega_n^q}\cdot\frac{K^2}{2R^n (K^2+a_\lambda
       R^n)^{q-1}}}{\sqrt{1+K^{\frac{2}{n}-2}\cdot\frac
       {m^2}{\omega_n^2}}}\cdot B^{\frac{1}{2}-\frac{1}{2n}}
       (t)\Biggr\} \notag
       \\
       &\leq \frac{m}{\omega_n}\cdot\frac{a_\lambda}
       {N^2(t)}\cdot(R^n-s)\cdot \Biggl\{ -\frac{(a_\lambda 
       +b_\lambda)K}{\sqrt{B(t)}}\cdot B'(t)\notag
      \\
      &\quad\,-\frac{a_\lambda^{q-1}(nm)^q \chi K^2}
      {2(K^2+a_\lambda R^n)^{q-1}\omega_n^q R^n
      \sqrt{1+K^{\frac{2}{n}-2}\cdot \frac{m^2}
      {\omega_n^2}}}\cdot B^{\frac{1}{2}-\frac{1}{2n}}(t) \Biggr\}
       \notag 
       \\
       &= \frac{m}{\omega_n}\cdot\frac{a_\lambda}
       {N^2(t)}\cdot(R^n-s)\cdot \Biggl\{ -\frac{(a_\lambda 
       +b_\lambda)K}{\sqrt{B(t)}}\cdot B'(t)\notag
      -c_1 K^2 B^{\frac{1}{2}-\frac{1}{2n}}(t) \Biggr\}
 \end{align}
for all $t\in(0,T)$ and all $s\in(K\sqrt{B(t)},R^n)$. 
Thanks to \eqref{condi;B3}, 
we finally derive that
 \begin{align*}
     -\frac{(a_\lambda +b_\lambda)K}{\sqrt{B(t)}}\cdot 
     B'(t)-c_1K^2 B^{\frac{1}{2}-\frac{1}{2n}}(t)
     &=\frac{(a_\lambda +b_\lambda)K}{\sqrt{B(t)}}\cdot 
      \left\{-B'(t)-\frac{c_1 K}{(a_\lambda +b_\lambda)}
     B^{1-\frac{1}{2n}}(t)\right\}
     \\
     &\leq 0
 \end{align*}
for all $t \in(0,T)$ and all $s\in(K\sqrt{B(t)},R^n)$, 
which together with 
\eqref{inequ;Pwout} implies \eqref{ineq;Pwout} holds. 
\end{proof}

%
%
\subsection{Subsolution properties: Inner region.}\label{inner}

In this subsection 
 we will consider in the inner region. In the following 
lemma 
 we provide calculations of $\mathcal{P}w_{\rm in}$ 
and properties of the function $A$ defined as \eqref{def;A} constructing 
$\underline{w}$ in the corresponding region.   

%
%

\begin{lem}\label{tool;inner}
 Let $n \in \N$, $m>0$, $\lambda\in[\frac{1}{3},1]$, $K > 0$ 
 be such that $K \geq \sqrt{b_\lambda R^n}$, and $T>0$, and 
 let $B \in C^1([0,T))$ be positive 
 and fulfill \eqref{condi;B1} and 
 $K\sqrt{B(t)} < R^n$ for all $t\in[0,T)$. 
 Then the function $w_{\rm in}$ defined as  
 \eqref{def;win} satisfies that 
 \begin{align}\label{equ;Pwin1}
    (\mathcal{P}w_{\rm in})(s,t) = A'(t)\varphi(\xi) 
    + \frac{A(t)\varphi'(\xi)}{B(t)}\cdot\{-\xi B'(t)+J_1(s,t)
    +J_2(s,t)\}
 \end{align}
for all $t\in(0,T)$ and all $s\in(0,K\sqrt{B(t)})\setminus\{B(t)\}$, 
where $\xi = \xi(s,t) = \frac{s}{B(t)}$, $\mathcal{P}$ 
is given by \eqref{def;P}, and
 \begin{align}\label{def;J1}
   J_1(s,t) := -n^{p+1}\cdot\frac{\xi^{2-\frac{2}{n}}
   \varphi''(\xi)}{\sqrt{B^{\frac{4}{n}-2}(t)\varphi'(\xi)
   +n^2B^{\frac{2}{n}-2}(t)\xi^{2-\frac{2}{n}}\varphi''^2
   (\xi)}}\cdot\left(\frac{A(t)\varphi'(\xi)}{B(t)}\right)^{p-1}
 \end{align}
as well as
 \begin{align}\label{def;J2}
   J_2(s,t) := -n^q \chi \cdot\frac{A(t)\varphi(\xi)-
   \frac{\mu}{n}B(t)\xi}{\sqrt{1+B^{\frac{2}{n}-2}(t)
   \xi^{\frac{2}{n}-2}(A(t)\varphi(\xi)-\frac{\mu}{n}B(t)
   \xi)^2}}\cdot\left(\frac{A(t)\varphi'(\xi)}{B(t)}\right)^{q-1}
 \end{align}
for $t\in(0,T)$ and $s\in(0,K\sqrt{B(t)})\setminus\{B(t)\}$. 
Moreover, the function $A$ defined as \eqref{def;A} fulfills  
  \begin{align}\label{ineq;A}
   A'(t) \leq 0 
 \end{align}
for all $t\in(0,T)$\/{\rm ;} in particular, 
 \begin{align}\label{def;AT}
   A(t) \geq A_T := \frac{m}{\omega_n}\cdot \frac{1}
   {1+\frac{a_\lambda R^n}{K^2}}
 \end{align}
holds for all $t\in(0,T)$. 
\end{lem}
\begin{proof}
Aided by arguments similar to those in the proofs of 
\cite[Lemmas 3.4 and 3.5]{BW}, 
from straightforward calculations we can attain 
the conclusion of this lemma.
\end{proof}

%
%
 
\subsection{Subsolution properties: Very inner region.}\label{veryinner}
In this subsection we will consider the case that 
$\xi=\frac{s}{B(t)}\in(0,1)$ which means $0<s<B(t)$. 
In order to show $\mathcal{P}w_{\rm in}\leq 0$ 
we have to see that for some $C>0$ and some $\beta\geq0$,
 \begin{align}\label{est;J12}
    J_1(s,t)+J_2(s,t)\leq-CB^\beta(t)
 \end{align}
holds for all $t\in(0,T)$ and all $s\in(0,B(t))$ with $J_1$ and 
$J_2$ as in \eqref{def;J1} 
and \eqref{def;J2}, respectively. Thanks to the convexity of 
$\varphi$ for $\xi \in (0,1)$, we obtain the term $J_1$ is negative 
in this region. However, it seems to be difficult to show \eqref{est;J12} 
in the case that $p<q$. Indeed, when $\varphi(\xi)=\lambda\xi^2$ 
which is used in \cite{BW}, 
if $B(t)$ is close to $0$, then we obtain from an arguments that
 \[
   J_1(s,t)+J_2(s,t)\geq -C_1B^{1-\frac{1}{n}}(t)\xi^{1-\frac{1}{n}}
   \cdot\left(\frac{A(t)\varphi'(\xi)}{B(t)}\right)^{p-1}  +C_2B(t)\xi
   \cdot\left(\frac{A(t)\varphi'(\xi)}{B(t)}\right)^{q-1}\geq 0
 \] 
with some $C_1,C_2>0$. Thus, we modify a function 
$\varphi$ on $(0,1)$ from \cite{BW} to infer that even though $B(t)$ 
is suitably small, the term $A(t)\varphi(\xi)-\frac{\mu}{n}B(t)\xi$ 
is positive which means $J_2<0$ for $t\in(0,T)$ and $\xi \in(0,B(t))$.  
Furthermore, we divide the estimate for $\mathcal{P}w_{\rm in}$  
into the case $n=1$ and the case $n\geq 2$ to achieve our purpose.
%
%
 
\begin{lem}\label{tool;very inner}
 Let $n \in \N$, $m > 0$, $\lambda\in[\frac{1}{3},1]$, and 
 $K > 0$ be such that $K \geq \sqrt{b_\lambda R^n}$,  
 and let $B_0\in(0,1)$ be such that $K\sqrt{B_0}< R^n$
and
 \begin{align}\label{condi;B_0_1}
     B_0\leq \frac{K^2}{4(a_\lambda+b_\lambda)^2}   
 \end{align}
as well as
 \begin{align}\label{condi;B_0_2}
   B_0 < \frac{2\lambda nA_T}{e^d\mu}
 \end{align}
with $\mu$, $d$ and $A_T$ given by \eqref{def;mu}, 
\eqref{def;d} and \eqref{def;AT}, respectively.
 Then, under the condition that for some $T > 0$, $B\in C^1([0,T))$
  is a positive and nonincreasing function satisfying
   $B(0)\leq B_0$, the inequality 
 \begin{align}\label{keyineq1}
   A(t)\varphi(\xi)-\frac{\mu}{n}B(t)
    \xi > 
    0
 \end{align}
 holds for all $t\in (0,T)$ and all $s\in(0,B(t))$ with $\varphi$ and $A$ as in \eqref{def;phi} and \eqref{def;A}, 
 respectively.
  \end{lem} 
 \begin{proof}
We write $\xi=\frac{s}{B(t)}$ for $t\in(0,T)$ and $s \in(0,B(t))$. 
According to \eqref{def;AT} and $B(t)\leq B(0)\leq B_0$, we obtain from \eqref{def;phi} 
that
 \begin{align}\label{keyestimate1}
    A(t)\varphi(\xi)-\frac{\mu}{n}B(t)\xi&\geq A_T\cdot\varphi(\xi)
    -\frac{\mu}{n}B_0\cdot\xi
    \\ \nonumber
    &=\frac{\mu}{n}\xi \left\{\frac{nA_T}{\mu}\cdot\frac{\varphi(\xi)}{\xi}
    -B_0\right\}
    \\ \nonumber
    &=\frac{\mu}{n}\xi \left\{\frac{2\lambda nA_T}{e^d\mu}\cdot
    \frac{e^{d\xi}-1}{d
    \xi}-B_0\right\}        
 \end{align}
 for all $t\in(0,T)$ and all $s\in(0,B(t))$.
 Thanks to that  
 \begin{align}\label{est2}
 \dfrac{e^{d\xi}-1}{d
 \xi}\geq 1
 \end{align}
 for all $\xi\in(0,1)$, we infer from \eqref{condi;B_0_2} and 
 \eqref{keyestimate1} that
\begin{align*}
    A(t)\varphi(\xi)-\frac{\mu}{n}B(t)\xi&\geq \frac{\mu}{n}\xi 
    \left\{\frac{2\lambda nA_T}{e^d\mu}\cdot\frac{e^{d\xi}-1}{d
    \xi}-B_0\right\}      
    \\ \nonumber
    &\geq \frac{\mu}{n}\xi \left\{\frac{2\lambda nA_T}{e^d\mu}-B_0\right\}>0
 \end{align*}
 for all $t\in(0,T)$ and all $s\in(0,B(t))$.
\end{proof}
%
%
Invoking Lemma \ref{tool;very inner}, under the assumption 
that the function $B$ is small and satisfies a suitable inequality, 
we derive that $\underline{w}$ becomes 
a 
 subsolution of \eqref{lem;base}. 
First, 
we will note when $n=1$. In this case, thanks to definition of 
$\varphi$ on $(0,1)$, for some $C>0$ we can estimate that 
$J_1\leq-C$ and show the purpose of this subsection.
 \begin{lem}\label{estimate;very inner}
 Let $n=1$, $m > 0$, $\lambda\in[\frac{1}{3},1]$ and   
 $K > 0$ be such that $K \geq \sqrt{b_\lambda R^n}$,
 and let $B_0\in(0,1)$ be such that $K\sqrt{B_0}\leq R^n$, 
 \eqref{condi;B_0_1} and \eqref{condi;B_0_2} hold. 
 For some $T>0$, if $B\in C^1([0,T))$
satisfies that  
  \begin{align}\label{condi;B04}
    \begin{cases}
    B'(t)\geq -\dfrac{d}{\sqrt{d^2+1}}\cdot 
    \left(\dfrac{2\lambda A_T}{e^d}\right)^{p-1},
    \\[4mm]
    B(0)\leq B_0
    \end{cases}
  \end{align}
for all $t\in(0,T)$ 
 with $d$ and $A_T$ as in \eqref{def;d} and \eqref{def;AT}, 
respectively, then the function 
$w_{\rm in}$ defined as \eqref{def;win} has the property that
  \[\label{ineq;Pwin01}
     (\mathcal{P}w_{\rm in})(s,t) \leq 0
  \]
for all $t\in(0,T)$ and all $s\in (0,B(t))$.
\end{lem}

\begin{proof}
 Writing $\xi=\frac{s}{B(t)}$ for $t\in(0,T)$ and $s \in(0,B(t))$, 
 we establish that 
 \begin{align*}
      \frac{B^2(t)\varphi'^2(\xi)}{\varphi''^2(\xi)}=B^2(t)\cdot
      \frac{(2\lambda e^{d(\xi-1)})^2}{(2d\lambda e^{d(\xi-1)})^2}=B^2(t)
      \cdot\frac{1}{d^2}\leq \frac{1}{d^2}
 \end{align*}
 for all $t\in(0,T)$ and all $s\in(0,B(t))$. The inequality leads to that
\begin{align}\label{est;J1}
      -J_1(s,t)&=\frac{\varphi''(\xi)}{\sqrt{B^2(t)\varphi'^2(\xi)+\varphi''^2(\xi)}}
      \cdot\left(\frac{A(t)\varphi'(\xi)}{B(t)}\right)^{p-1}
      \\ \nonumber
     &\geq
     \frac{\varphi''(\xi)}{\sqrt{\left(\frac{1}{d^2}+1\right)
     \varphi''^2(\xi)}}\cdot \left(\frac{A(t)\varphi'(\xi)}{B(t)}\right)^{p-1}
      \\ \nonumber
     &=\frac{d}{\sqrt{d^2+1}}\cdot \left(\frac{A(t)\varphi'(\xi)}{B(t)}\right)^{p-1}
\end{align}
for all $t\in(0,T)$ and all $s\in(0,B(t))$.  Thanks to \eqref{def;AT} 
and $\varphi'(\xi)\geq \varphi'(0)=\frac{2\lambda}{e^d}$, 
we infer from $B(t)<1$ that
\begin{align}\label{est1}
      \left(\frac{A(t)\varphi'(\xi)}{B(t)}\right)^{p-1}\geq 
      \left(\frac{A_T\cdot \frac{2\lambda}{e^d}}{1}\right)^{p-1}
      =\left(\frac{2\lambda A_T}{e^d}\right)^{p-1}
\end{align}
for all $t\in(0,T)$ and all $s\in(0,B(t))$. On the other hand, 
using \eqref{keyineq1}, we have 
\begin{align}\label{est;J2}
      J_2< 
      0
\end{align} for all $t\in(0,T)$ and all $s\in(0,B(t))$.
 Thus, plugging \eqref{est;J1} and \eqref{est;J2} into \eqref{equ;Pwin1}, 
 we derive from 
 \eqref{condi;B04}, \eqref{est1} 
 and the inequality $\xi\leq1$ 
 that
\begin{align*}
     \frac{B(t)}{A(t)\varphi'(\xi)}(\mathcal{P} w_{\rm in})(s,t)
    &\leq -\xi B'(t)+J_1(s,t)+J_2(s,t)
    \\ \nonumber
    &\leq - \xi B'(t)-\frac{d}{\sqrt{d^2+1}}\cdot \left(\frac{A(t)
    \varphi'(\xi)}{B(t)}\right)^{p-1}
        \\ \nonumber
    &\leq - \xi B'(t)-\frac{d}{\sqrt{d^2+1}}\cdot 
    \left(\frac{2\lambda A_T}{e^d}\right)^{p-1}
    \\ \nonumber
    &\leq \xi\left\{-B'(t)-\frac{d}{\sqrt{d^2+1}}\cdot 
    \left(\frac{2\lambda A_T}{e^d}\right)^{p-1}\right\}
    \\ \nonumber
    &\leq 0
\end{align*}
for all $t\in(0,T)$ and all $s\in(0,B(t))$, and it concludes the proof.
\end{proof}
%
%
In the case $n\geq2$,  
if the function $B$ 
 is suitably small, for some $C>0$ we obtain that 
$J_2\leq -C$ and it leads to achievement of the purpose.
\begin{lem}\label{estimate2;veryinner}
 Let $n \geq 2$, $m > 0$, $\lambda\in[\frac{1}{3},1]$, 
 and $K > 0$ be such that $K \geq \sqrt{b_\lambda R^n}$, and let 
 $B_0\in(0,1)$ be such that 
 $K\sqrt{B_0}\leq R^n$, \eqref{condi;B_0_1}, \eqref{condi;B_0_2} and 
 \begin{align}\label{condi;B7}
           B_0^{2-\frac{2}{n}} \leq \left(\frac{2\lambda A_T}{e^d}
           -\frac{\mu}{n}B_0\right)^2
 \end{align}
  hold. For some 
  $T>0$, if $B\in C^1([0,T))$ satisfies that  
  \begin{align}\label{condi;B4}
    \begin{cases}
    B'(t)\geq -\dfrac{n^q \chi }{\sqrt{2}}\cdot
    \left(\dfrac{2\lambda A_T}{e^d}\right)^{q-1}B^{1-\frac{1}{n}}(t),
    \\[4mm]
    B(0)\leq B_0
    \end{cases}
  \end{align}
for all $t\in(0,T)$ 
 with $d$ and $A_T$ as in \eqref{def;d} and 
\eqref{def;AT}, respectively, then the function $w_{\rm in}$ 
defined as \eqref{def;win} has the property that
  \[\label{ineq;Pwin1}
     (\mathcal{P}w_{\rm in})(s,t) \leq 0
  \]
for all $t\in(0,T)$ and all $s\in (0,B(t))$.
\end{lem}
\begin{proof}
We write $\xi=\frac{s}{B(t)}$ for $t\in(0,T)$ and 
$s \in(0,B(t))$. Thanks to \eqref{condi;B7} and 
\[
B(t)\leq B(0) \leq B_0,
\] 
we obtain that
 \begin{align*} 
             B^{2-\frac{2}{n}}(t) \leq B_0^{2-\frac{2}{n}}
             &
             \leq 
             \left(\frac{2\lambda A_T}{e^d}-\frac{\mu}{n}B_0\right)^2
\\
             &\leq 
             \left(\frac{2\lambda A_T}{e^d}-\frac{\mu}{n}B(t)\right)^2
 \end{align*}
 for all $t\in(0,T)$ and all $s\in (0,B(t))$. Moreover, 
 since \eqref{def;AT} and \eqref{est2} hold, 
 we can see from \eqref{def;phi}, \eqref{est2} and the inequality $\xi<1$ 
 that 
 \begin{align*}
            B^{2-\frac{2}{n}}(t) &\leq  \left(\frac{2\lambda A_T}{e^d}\cdot 1    
            -\frac{\mu}{n}B(t)\right)^2
            \\ \nonumber
           &\leq \left(A(t)\cdot\frac{2\lambda}{e^d}\cdot
           \frac{e^d\xi-1}{d
           \xi}-\frac{\mu}{n}B(t)\right)^2
            \\ \nonumber
           &=\xi^{-2}\left(A(t)\varphi(\xi)-\frac{\mu}{n}B(t)\xi\right)^2
 \end{align*}  
 for all $t\in(0,T)$ and all $s\in (0,B(t))$. 
 Therefore, we derive that
\begin{align}\label{est3}
           1\leq B^{\frac{2}{n}-2}\xi^{-2}\left(A(t)\varphi(\xi)
           -\frac{\mu}{n}B(t)\xi\right)^2
\end{align}
for all $t\in(0,T)$ and all $s\in (0,B(t))$. Thanks to \eqref{est3}, 
we infer from $\xi<1$ that
 \begin{align}\label{ineq;J22}
           -J_2(s,t) &= n^q \chi \cdot\frac{A(t)\varphi(\xi)-
           \frac{\mu}{n}B(t)\xi}{\sqrt{1+B^{\frac{2}{n}-2}(t)
           \xi^{\frac{2}{n}-2}(A(t)\varphi(\xi)-\frac{\mu}{n}B(t)
           \xi)^2}}\cdot\left(\frac{A(t)\varphi'(\xi)}{B(t)}\right)^{q-1}
           \\ \nonumber
          &\geq n^q \chi \cdot\frac{A(t)\varphi(\xi)-
         \frac{\mu}{n}B(t)\xi}{\sqrt{(1+\xi^{\frac{2}{n}})B^{\frac{2}{n}-2}(t)
         \xi^{-2}(A(t)\varphi(\xi)-\frac{\mu}{n}B(t)
         \xi)^2}}\cdot\left(\frac{A(t)\varphi'(\xi)}{B(t)}\right)^{q-1}
         \\ \nonumber
          &\geq n^q \chi \cdot\frac{1}{\sqrt{(1+1^{\frac{2}{n}})B^{\frac{2}{n}-2}(t)
         \xi^{-2}}}\cdot\left(\frac{A(t)\varphi'(\xi)}{B(t)}\right)^{q-1}
         \\ \nonumber
         &= \frac{n^q \chi }{\sqrt{2}}B^{1-\frac{1}{n}}(t)
         \xi\left(\frac{A(t)\varphi'(\xi)}
         {B(t)}\right)^{q-1}
  \end{align} 
  for all $t\in(0,T)$ and all $s\in(0,B(t))$.
   Moreover, using \eqref{def;AT} and $\varphi'(\xi)\geq \varphi'(0)=
   \frac{2\lambda}{e^d}$, we obtain from $B(t)<1$ that
\begin{align}\label{est11}
   \left(\frac{A(t)\varphi'(\xi)}{B(t)}\right)^{q-1}\geq 
   \left(\frac{A_T\cdot \frac{2\lambda}{e^d}}{1}\right)^{q-1}
   =\left(\frac{2\lambda A_T}{e^d}\right)^{q-1}.
 \end{align}
  On the other hand, from the fact $\varphi',\varphi''>0$ 
  (given by \eqref{def;phi'} 
  and \eqref{def;phi''}, respectively), we verify that
 \begin{align}\label{est;J11}
   J_1< 0
 \end{align}  
  for all $t\in(0,T)$ and all $s\in(0,B(t))$.
 Recalling \eqref{equ;Pwin1} and \eqref{ineq;A} 
 (see Lemma \ref{tool;inner}), we can show that a combination of 
 \eqref{ineq;J22} and \eqref{est;J11} yields that by 
 \eqref{condi;B4} and \eqref{est11},
 \begin{align*}
     \frac{B(t)}{A(t)\varphi'(\xi)}(\mathcal{P} w_{\rm in})(s,t)
    &\leq 
     - \xi B'(t)-\frac{n^q \chi }{\sqrt{2}}\cdot\left(\frac{A(t)
     \varphi'(\xi)}{B(t)}\right)^{q-1}B^{1-\frac{1}{n}}(t)\xi
     \\ 
    &\leq
    \xi\left\{-B'(t)-\frac{n^q \chi }{\sqrt{2}}\cdot
    \left(\frac{2\lambda A_T}{e^d}\right)^{q-1}B^{1-\frac{1}{n}}(t)
    \right\}
    \\
    &\leq 0
 \end{align*}
  for all $t\in(0,T)$ and all $s\in(0,B(t))$, 
  which means the end of the proof.
\end{proof}
%
%

\subsection{Subsolution properties: Intermediate region.}\label{intermadiate}

%
%
In this subsection we will consider the
case that $s\in(B(t),K\sqrt{B(t)})$. Here the term 
$J_1$ is positive due to the definition of $\varphi$ in the region. 
Therefore, an estimate for $J_2$ is important in this part. 
The following arguments are based on these of \cite{BW}. 
\begin{lem}\label{tool1;intermediate}
 Let $n \in \N$, $m > 0$, $K > 0$ and $T >0$, 
 and let $B \in C^1([0,T))$ be positive 
 and satisfy that \eqref{condi;B1} and 
 $K\sqrt{B(t)} < R^n$ for all $t\in[0,T)$. 
 Then the function 
 $J_1$ defined as \eqref{def;J1} satisfies 
 \begin{align}\label{estimate;J1}
    J_1(s,t) \leq n^p B^{1-\frac{1}{n}}(t)\xi^{1-\frac{1}{n}}
    \left(\frac{A(t)\varphi'(\xi)}{B(t)}\right)^{p-1}
 \end{align}
for all $t\in(0,T)$ and all $s\in
    (B(t),K\sqrt{B(t)})$ with $\xi = \frac{s}{B(t)}$.
\end{lem}
\begin{proof}
An argument similar to that in the proof of 
\cite[Lemma 3.7]{BW} implies the conclusion this lemma.
\end{proof}

%
%
The following two lemmas have already been proved in 
proofs of \cite[Lemmas 3.8 and 3.10]{BW}. 
Thus we only recall statements of lemmas.
\begin{lem}\label{tool2;intermediate}
  Let $n \in \N $, $m > 0$, $\lambda\in[\frac{1}{3},1]$, $K > 0$ 
  be such that $K \geq \sqrt{b_\lambda R^n}$, 
  and let $B_0\in(0,1)$ satisfy $K\sqrt{B_0} < R^n$ 
  and \eqref{condi;B_0_1}, 
  as well as for some $T > 0$ let $B \in C^1([0,T))$ 
   be positive and nonincreasing and be such that 
   $B(0) \leq B_0$. Then the inequality
 \begin{align}\label{inequ4}
    \frac{1}{A^2(t)B^{\frac{2}{n}-2}(t)\xi^{\frac{2}{n}-2}
    \varphi^2(\xi)}\leq\frac{\omega_n^2}{\lambda^2m^2}
    \cdot\left(1+\frac{a_\lambda R^n}{K^2}\right)
    \cdot K^{2-\frac{2}{n}}B_0^{3-\frac{3}{n}}
 \end{align} 
 holds for all $t\in(0,T)$ and all $s\in(B(t),K\sqrt{B(t)})$ with $\xi = \frac{s}{B(t)}$. 
\end{lem}
%
%
\begin{lem}\label{tool3;intermediate}
  Let $n \in\N$, $m > 0$, $\lambda\in[\frac{1}{3},1]$, $K > 0$, 
  $\delta\in(0,1)$, and let $B_0\in(0,1)$ fulfill 
  $K\sqrt{B_0} < R^n$, \eqref{condi;B_0_1}
  and
 \[\label{condi;B_0_5}
    \frac{\mu}{nA_T}\cdot {\rm max}\left\{\frac{B_0}
    {\lambda},\,2K\sqrt{B_0}\right\} \leq \delta
 \]
  with $\mu$ and $A_T$ introduced in \eqref{def;mu} 
  and \eqref{def;AT}, respectively,  and suppose that $T > 0$ 
  and $B\in C^1([0,T))$ is positive 
  and such that
 \[\label{condi;B_0_6}
    B(t) \leq B_0
 \]  
for all $t\in(0,T)$. Then the inequality
 \begin{align}\label{inequ5}
    A(t)\varphi(\xi)-\frac{\mu}{n}B(t)
    \xi \geq (1-\delta)A(t)\varphi(\xi)
 \end{align}  
 holds for all $t\in(0,T)$ and all $s\in((B(t),K\sqrt{B(t)})$ 
 with $\xi=\frac{s}{B(t)}$ for $s\in (B(t),K\sqrt{B(t)})$ and $t\in(0,T)$.
\end{lem}

%
%

In order to have the estimate for $\mathcal{P}w_{\rm in}$ 
in the intermediate region when $p<q$, we will provide 
the following lemma.
\begin{lem}\label{tool4;intermediate}
  Let $n \in \N$, $m > 0$, $\lambda\in[\frac{1}{3},1]$, $K > 0$, 
  and let $B \in C^1([0,T))$ be such that $K\sqrt{B(t)} < R^n$ 
  and be a positive and nonincreasing function and put
 \begin{align}\label{def;sigma}
    \sigma := \frac{m}{\omega_n}\cdot\frac{a_\lambda}
    {K^2+a_\lambda R^n}.
 \end{align}  
Then the inequality 
 \begin{align}\label{keyestimate3} 
    \left(\frac{A(t)\varphi'(\xi)}{B(t)}\right)^{-k}\leq 
    \sigma^{-k}
 \end{align}
holds for all $k \geq 0$ and all $s\in(B(t),K\sqrt{B(t)})$, 
$t\in(0,T)$ with $\xi = \frac{s}{B(t)}$, 
where 
$A$ and $\varphi'$ are 
as in \eqref{def;A} and \eqref{def;phi'}, respectively.
\end{lem} 
 \begin{proof} 
We write $\xi = \frac{s}{B(t)}$ for all $t\in(0,T)$ 
 and all $s\in(B(t),K\sqrt{B(t)})$. Using that 
\[ 
  \varphi'(\xi) = 
  \frac{a_\lambda}{(\xi-b_\lambda)^2} 
\] 
for $ 1< \xi < \frac{K}{\sqrt{B(t)}}$, 
we infer from $B(t)<1$ and $0\leq b_\lambda\leq 1$ for 
$\lambda\in[\frac{1}{3},1]$ that
 \begin{align}\label{estimate;phi'1}
     \varphi'(\xi) &= \frac{a_\lambda}{(\xi-b_\lambda)^2} =
      \frac{a_\lambda}{\xi^2-2\xi b_\lambda+b_\lambda^2} 
     \\ \nonumber
      &\geq
       \frac{a_\lambda}{\bigl(\frac{K}{\sqrt{B(t)}}\bigr)^2-2\cdot1
       \cdot b_\lambda+b_\lambda^2}
     \\
    &=\frac{a_\lambda B(t)}{K^2+(b_\lambda^2-2b_
    \lambda)B(t)} \geq \frac{a_\lambda}{K^2}B(t).\notag
 \end{align}
 A combination of \eqref{def;AT} and 
 \eqref{estimate;phi'1} yields that
 \begin{align*}
     \left(\frac{A(t)\varphi'(\xi)}{B(t)}\right)&\geq
     \frac{m}{\omega_n}\cdot\frac{1}{1+\frac{a_\lambda R^n}
     {K^2}}\cdot\frac{a_\lambda}{K^2} 
     \\
     &= \frac{m}{\omega_n}
     \cdot\frac{a_\lambda}{K^2+a_\lambda R^n} = \sigma 
 \end{align*}
 for all $t\in(0,T)$ and all $s\in(B(t),K\sqrt{B(t)})$, 
 which means the end of proof. 
 \end{proof}

%
%

Thanks to above lemmas, we can see that if a function $B$  
is suitably small and nonincreasing and fulfills some differential 
inequality, then $\underline{w}$ becomes 
 a subsolution of 
\eqref{lem;base} in the 
intermediate region $(B(t),K\sqrt{B(t)})$.

First, we will consider the case $n=1$. The largeness 
condition for $m$ (see \eqref{condi;mass}) will lead 
 to 
the following lemma.

\begin{lem}\label{estimate1;intermediate}
 Let $p\leq q$, $n = 1$, $\chi > 0$ and $\lambda \in 
 (\frac{5-\sqrt{17}}{2},1]$, and put $\delta_\lambda := 
 \frac{a_\lambda}{b_\lambda}$ with $a_\lambda$ 
 and $b_\lambda$ as in \eqref{def;ab}, and let 
 $m_c(p,q,\chi,\lambda,R)$ be such that 
\begin{align}\label{condi;mc}
    m_c:=\inf \left\{m\  \Bigg|\  \exists\, \lambda \in 
 \left(\frac{5-\sqrt{17}}{2},1\right];\,\frac{(1-\delta_\lambda)m \chi}
 {\sqrt{\frac{1+\delta_\lambda}
    {\lambda^2}+ m^2}}-\left(\frac{m}{\omega_n}\cdot 
    \frac{a_\lambda}{(a_\lambda + b_\lambda)R}\right)^{p-q} =0\right\}.
\end{align}
 Then for all $m>m_c$ there exist $K > 0$, $\kappa_1 > 0$, 
 and $B_{01}$ such that $K\sqrt{B_{01}} < R^n$, 
 and for some $T > 0$, if $B\in C^1([0,T))$ 
 is a positive and nonincreasing function fulfilling \eqref{condi;B1} 
 as well as
  \begin{align}\label{condi;B5}
    \begin{cases}
    B'(t)\geq -\kappa_1\sqrt{B(t)},
    \\[1mm] 
    B(0)\leq B_{01}
    \end{cases}
  \end{align}
  for all $t\in(0,T)$, 
   then $w_{\rm in}$ as in \eqref{def;win} satisfies
  \[\label{ineq;Pwin4}
     (\mathcal{P}w_{\rm in})(s,t) \leq 0
  \]
for all $t\in(0,T)$ and all $s\in(B(t),K\sqrt{B(t)})$.
  \end{lem}
  \begin{proof}
The proof is based on that of \cite[Lemma 3.11]{BW}.
Since  $\delta_\lambda = \frac{a_\lambda}{b_\lambda}$ holds, 
we obtain $b_\lambda R = \frac{a_\lambda R}{\delta_\lambda}$ 
with $a_\lambda$ and $b_\lambda$ 
as in \eqref{def;ab}, 
which enables us to pick 
$K > 0$ such that
 \begin{align}\label{estimate2}
  K\geq\sqrt{b_\lambda R}
 \end{align}
and
 \begin{align}\label{estimate3}
  \frac{a_\lambda R}{K^2} \leq\delta_\lambda
 \end{align}
as well as
\begin{align}\label{def;c2}
   c_1 &:= \frac{(1-\delta_\lambda)m \chi}{\sqrt{\frac{1+
   \delta_\lambda}{\lambda^2}+ m^2}}-\left(\frac{m}
   {\omega_n}\cdot \frac{a_\lambda}{K^2 + a_\lambda R}
   \right)^{p-q}
   \\ \notag
   &\,=\frac{(1-\delta_\lambda)m \chi}{\sqrt
   {\frac{1+\delta_\lambda}{\lambda^2}+ m^2}}
   -\sigma^{p-q}
\end{align}
is positive (by \eqref{condi;mc} and $m>m_c$) with $\sigma$ 
as in \eqref{def;sigma}. 
Furthermore, it is possible to fix 
$B_{01}\in(0,1)$ fulfilling $K\sqrt{B_{01}} < R$ and 
 \begin{align}\label{condi;B_0_7}
    B_{01}\leq \frac{K^2}{4(a_\lambda+b_\lambda)^2}
 \end{align}
 as well as
  \begin{align}\label{condi;B_0_8}
    \frac{\mu}{nA_T}\cdot {\rm max}\left\{\frac{B_{01}}
    {\lambda},\,2K\sqrt{B_{01}}\right\} \leq \delta_\lambda
 \end{align}
 with $A_T$ as in \eqref{def;AT}, and put
 \begin{align}\label{def;kappa1}
    \kappa_1  := \frac{\sigma^{q-1} c_1}{K},
 \end{align}
and let $T  > 0$ and $B\in C^1([0,T))$ 
be positive and nonincreasing and such that
 \eqref{condi;B5} holds.
Then, recalling $A' \leq 0$ on $(0,T)$ (see Lemma \ref{tool;inner}), 
we infer from   
 \eqref{equ;Pwin1} that
 \begin{align}\label{ineq;Pwin6}
     \frac{B(t)}{A(t)\varphi'(\xi)}\cdot(\mathcal{P} 
     w_{\rm in})(s,t)
      \leq -\xi B'(t)+J_1(s,t)+J_2(s,t)
 \end{align}
for all $t\in(0,T)$ and all $s\in(B(t),K\sqrt{B(t)})$  
with $\xi =\frac{s}{B(t)}$ and $J_1$ and $J_2$ 
as in \eqref{def;J1} and \eqref{def;J2}. 
Here, according to Lemma \ref{tool1;intermediate}, 
we obtain from \eqref{estimate;J1} that
 \begin{align}\label{estimate;J11}
    J_1(s,t) \leq 
    \left(\frac{A(t)\varphi'(\xi)}{B(t)}\right)^{p-1}
 \end{align}
 for all $t\in(0,T)$ and all $s\in(B(t),K\sqrt{B(t)})$. Since \eqref{estimate2}, 
 \eqref{estimate3},  \eqref{condi;B_0_7} and \eqref{condi;B_0_8} hold, 
a combination of Lemmas \ref{tool2;intermediate} 
and \ref{tool3;intermediate} with an argument 
 similar to that in 
  the proof of \cite[Lemma 3.11]{BW} leads to that
 \[\label{estimate4}
     \sqrt{1+\left(A(t)\varphi(\xi)-\mu B(t)\xi\right)^2}
     \leq \sqrt{\frac{1+\delta_\lambda}{\lambda^2 m^2}+1}\cdot A(t)
     \varphi(\xi) 
 \]
 for all $t\in(0,T)$ 
and all $s\in(B(t),K\sqrt{B(t)})$. 
Therefore, we can see from \eqref{inequ5} 
that 
 \begin{align}\label{estimate;J2}
   -J_2(s,t) &= \chi \cdot\frac{A(t)\varphi(\xi)-
   \mu B(t)\xi}{\sqrt{1+(A(t)\varphi(\xi)-\mu B(t)
   \xi)^2}}\cdot\left(\frac{A(t)\varphi'(\xi)}{B(t)}\right)^{q-1}
   \\ \nonumber
   &\geq  \chi \cdot\frac{(1-\delta_\lambda)A(t)\varphi(\xi)}
   {\sqrt{\frac{1+\delta_\lambda}{\lambda^2m^2}+1}\cdot A(t)
   \varphi(\xi)}\cdot\left(\frac{A(t)\varphi'(\xi)}{B(t)}
   \right)^{q-1}
   \\ \nonumber
   &= \frac{(1-\delta_\lambda)m\chi}{\sqrt{\frac{1+\delta_\lambda}
   {\lambda^2}+m^2}}\cdot\left(\frac{A(t)\varphi'(\xi)}
   {B(t)}\right)^{q-1}
 \end{align}
for all $t\in(0,T)$ and  all $s\in(B(t),K\sqrt{B(t)})$. 
From the relation $p\leq q$, a combination of  \eqref{keyestimate3} and \eqref{def;kappa1}--\eqref{estimate;J2}, 
 along with 
 the definition of $c_1$ (see \eqref{def;c2}) 
implies that 
 \begin{align}\label{ineq;Pwin7}
     &\frac{B(t)}{A(t)\varphi'(\xi)}\cdot(\mathcal{P} 
     w_{\rm in})(s,t)
     \\ \nonumber
      &\leq -\xi B'(t)+\left(\frac{A(t)\varphi'(\xi)}{B(t)}
      \right)^{p-1}-\frac{(1-\delta_\lambda)m\chi}{\sqrt{\frac{1+\delta_\lambda}
      {\lambda^2}+m^2}}\cdot\left(\frac{A(t)\varphi'(\xi)}
      {B(t)}\right)^{q-1}
      \\ \nonumber
      &=\left(\frac{A(t)\varphi'(\xi)}{B(t)}\right)^{q-1}\Biggl
      \{-\xi B'(t)\left(\frac{A(t)\varphi'(\xi)}{B(t)}\right)^{1-q}
      +\left(\frac{A(t)\varphi'(\xi)}{B(t)}\right)^{p-q}-
      \frac{(1-\delta_\lambda)m\chi}{\sqrt{\frac{1+\delta_\lambda}
      {\lambda^2}+m^2}}\Biggr\}
      \\ \nonumber
      &\leq \left(\frac{A(t)\varphi'(\xi)}{B(t)}\right)^{q-1}
      \Biggl\{-\xi B'(t)\sigma^{1-q}+\sigma^{p-q}-
      \frac{(1-\delta_\lambda)m\chi}{\sqrt{\frac{1+\delta_\lambda}{
      \lambda^2}+m^2}}\Biggr\}
      \\ \nonumber
      &= \left(\frac{A(t)\varphi'(\xi)}{B(t)}\right)^{q-1}
      \sigma^{1-q}\left(-\xi B'(t)-\sigma^{q-1}c_1\right)
 \end{align}
for all $t\in(0,T)$ and all $s\in(B(t),K\sqrt{B(t)})$. 
Recalling that $\xi \leq \frac{K}{\sqrt{B(t)}}$ holds in the region, 
we infer from the definition of $\kappa_1$ (see \eqref{def;kappa1}) that
 \begin{align}\label{est5}
   -\xi B'(t)-\sigma^{q-1}c_1 &= \xi \cdot\left(-B'(t)
   -\frac{\sigma^{q-1}c_1}{\xi}\right)
   \\ \nonumber
   &\leq \xi \cdot\left(-B'(t)-\frac{\sigma^{q-1}c_1
   \sqrt{B(t)}}{K}\right)
   \\ \nonumber
   & = \xi \cdot\left(-B'(t)-\kappa_1\sqrt{B(t)}\right)
   \\ \nonumber
   &\leq 0
 \end{align}
for all $t\in(0,T)$ and all $s\in(B(t),K\sqrt{B(t)})$. 
A combination of  \eqref{ineq;Pwin7} and \eqref{est5} leads to 
that we can attain the conclusion of  the proof.
\end{proof}

%
%

Second, we will consider the case $n\geq 2$. 
Due to choosing 
suitably $\chi$ as in \eqref{condi;chi}, we will infer from 
the condition for a function $B$ that $\underline{w}$ 
becomes 
 a subsolution of \eqref{lem;base}.

\begin{lem}\label{estimate2;intermediate}
 Let $p\leq q$, $n \geq 2 $, $m > 0$, $\chi>0$ and $\lambda = \frac{1}{3}$. 
 Then for all $\chi > (\frac{mn}
 {\omega_n R^n})^{p-q}$ there exist $K > 0$, $\kappa_n > 0$, and $B_{0n}$ 
 such that $K\sqrt{B_{0n}} < R^n$, and such that if 
 $T > 0$ and $B\in C^1([0,T))$ is a positive and 
 nonincreasing such that
 \begin{align}\label{condi;B6}
    \begin{cases}
    B'(t)\geq -\kappa_n B^{1-\frac{1}{2n}}(t),
    \\[1mm] 
    B(0)\leq B_{0n}
    \end{cases}
  \end{align}
for all $t\in(0,T)$, 
then the function $w_{\rm in}$ defined in \eqref{def;win} 
satisfies
\[\label{ineq;Pwin5}
     (\mathcal{P}w_{\rm in})(s,t) \leq 0
  \]
  for all $t\in(0,T)$ and all $s\in(B(t),K\sqrt{B(t)})$.
\end{lem}
\begin{proof}
The proof is based on that of \cite[Lemma 3.12]{BW}.
Thanks to $\lambda = \frac{1}{3}$ and 
$\chi > (\frac{mn}{\omega_n R^n})^{p-q}$,  
we have that $b_\lambda=0$ and
\[
 \chi > \left(n \cdot \frac{m}{\omega_n }\cdot \frac
 {a_\lambda}{a_\lambda R^n}\right)^{p-q}
\] 
with $a_\lambda$ as in \eqref{def;ab}. 
Then there exists some $K > 0$ fulfilling
\begin{align}\label{estimate6}
 \chi &> \left(n \cdot \frac{m}{\omega_n }\cdot \frac
 {a_\lambda}{K^2 + a_\lambda R^n}\right)^{p-q}  
 \\ \nonumber
 &=(n\sigma)^{p-q}
\end{align}
with $\sigma$ given by \eqref{def;sigma}. 
Using \eqref{estimate6}, we can choose $\delta\in(0,1)$ 
suitably small such that
 \begin{align}\label{def;c3}
    c_1 := n^q\cdot\left\{\frac{(1-\delta)\chi}{\sqrt
    {1+\delta}}-(n\sigma)^{p-q}\right\}
 \end{align}
is positive. Finally, we take $B_{0n}\in(0,1)$ such that 
$K\sqrt{B_{0n}} <  R^n$ and
 \begin{align}\label{condi;B_0_9}
    B_{0n}\leq \frac{K^2}{4(a_\lambda+b_\lambda)^2}
 \end{align} 
as well as
   \begin{align}\label{condi;B_0_10}
    \frac{\mu}{nA_T}\cdot {\rm max}\left\{\frac{B_{0n}}
    {\lambda},\,2K\sqrt{B_{0n}}\right\} \leq \delta
 \end{align}
with $A_T$ as in \eqref{def;AT} and 
 \begin{align}\label{inequ6}
    \frac{\omega_n^2}{\lambda^2m^2}
    \cdot\left(1+\frac{a_\lambda R^n}{K^2}\right)
    \cdot K^{2-\frac{2}{n}}B_{0n}^{3-\frac{3}{n}}
    \leq\delta,
 \end{align}
 and we put
 \begin{align}\label{def;kappan}
    \kappa_n  := \sigma^{q-1} c_1 K^{-\frac{1}{n}},
 \end{align}
and suppose $T  > 0$ and $B\in C^1([0,T))$ 
is positive and nonincreasing and such that \eqref{condi;B6} holds.
Then, from
\eqref{condi;B_0_9},
Lemma \ref{tool;inner} 
implies 
that $A' \leq 0$ on $(0,T)$. Recalling \eqref{equ;Pwin1},
 we derive that
 \begin{align}\label{ineq;Pwin8}
     \frac{B(t)}{A(t)\varphi'(\xi)}\cdot(\mathcal{P}
      w_{\rm in})(s,t)
      \leq -\xi B'(t)+J_1(s,t)+J_2(s,t)
 \end{align} 
 for all $t\in(0,T)$ and all $s\in(B(t),K\sqrt{B(t)})$ 
 with $\xi = \frac{s}{B(t)}$, and $J_1$ and $J_2$ 
 given by \eqref{def;J1} and \eqref{def;J2}, 
 respectively. From \eqref{condi;B_0_9} 
 and \eqref{condi;B_0_10}, using Lemma \ref{tool3;intermediate}, 
 we obtain that 
  \begin{align}\label{inequ7}
    A(t)\varphi(\xi)-\frac{\mu}{n}B(t)
    \xi \geq (1-\delta)A(t)\varphi(\xi)
 \end{align}  
 for all $t\in(0,T)$ and all $s\in(B(t),K\sqrt{B(t)})$.
 Noticing \eqref{condi;B_0_9}, we can use  
 Lemma \ref{tool2;intermediate} to show from 
 \eqref{inequ6} that
 \begin{align}\label{inequ8}
    \frac{1}{A^2(t)B^{\frac{2}{n}-2}(t)\xi^{\frac{2}{n}-2}
    \varphi^2(\xi)}&\leq \frac{\omega_n^2}{\lambda^2m^2}
    \cdot\left(1+\frac{a_\lambda R^n}{K^2}\right)
    \cdot K^{2-\frac{2}{n}}B_{0n}^{3-\frac{3}{n}}
    \\ \nonumber
    &\leq\delta
 \end{align}
 for all $t\in(0,T)$ and all $s\in(B(t),K\sqrt{B(t)})$. 
 Employing \eqref{inequ7}, \eqref{inequ8} 
 and the fact that $\delta < 1$, we can estimate that
 \begin{align}\label{estimate;J22}
   -J_2(s,t) &= n^q\chi \cdot\frac{A(t)\varphi(\xi)-
   \frac{\mu}{n} B(t)\xi}{\sqrt{1+B^{\frac{2}{n}-2}(t)
   \xi^{\frac{2}{n}-2}(A(t)\varphi(\xi)-\frac{\mu}{n} B(t)
   \xi)^2}}\cdot\left(\frac{A(t)\varphi'(\xi)}{B(t)}\right)^{q-1}
   \\ \nonumber
   &\geq  n^q\chi \cdot\frac{(1-\delta)A(t)\varphi(\xi)}
   {\sqrt{ 1+B^{\frac{2}{n}-2}(t)\xi^{\frac{2}{n}-2}A^2(t)
   \varphi^2(\xi)}}\cdot\left(\frac{A(t)\varphi'(\xi)}
   {B(t)}\right)^{q-1}
   \\ \nonumber
   &\geq n^q\chi \cdot\frac{(1-\delta)A(t)\varphi(\xi)}
   {\sqrt{ (\delta+1)B^{\frac{2}{n}-2}(t)\xi^{\frac{2}{n}-2}
   A^2(t)\varphi^2(\xi)}}\cdot\left(\frac{A(t)\varphi'(\xi)}
   {B(t)}\right)^{q-1}
   \\ \nonumber
   & = \frac{(1-\delta)n^q\chi}{\sqrt{\delta+1}}\cdot
    B^{1-\frac{1}{n}}(t)\xi^{1-\frac{1}{n}}\cdot
    \left(\frac{A(t)\varphi'(\xi)}{B(t)}\right)^{q-1}
 \end{align}
for all $t\in(0,T)$ and all $s\in(B(t),K\sqrt{B(t)})$. 
On the other hand, we use Lemma \ref{tool1;intermediate} to
 show that
 \begin{align}\label{estimate;J12}
    J_1(s,t) \leq n^p B^{1-\frac{1}{n}}(t)\xi^{1-\frac{1}{n}}
    \left(\frac{A(t)\varphi'(\xi)}{B(t)}\right)^{p-1}
 \end{align}
 for all $t\in(0,T)$ and all $s\in(B(t),K\sqrt{B(t)})$. Thanks to \eqref{ineq;Pwin8}, \eqref{estimate;J22} and \eqref{estimate;J12}, we derive that
  \begin{align}\label{ineq;Pwin9}
     \frac{B(t)}{A(t)\varphi'(\xi)}\cdot(\mathcal{P} 
     w_{\rm in})(s,t)&\leq -\xi B'(t)+n^p B^{1-\frac{1}{n}}(t)
      \xi^{1-\frac{1}{n}}\left(\frac{A(t)\varphi'(\xi)}
      {B(t)}\right)^{p-1}
      \\ \nonumber
      &\quad\,-\frac{(1-\delta)n^q\chi}{\sqrt{\delta+1}}
            \cdot B^{1-\frac{1}{n}}(t)\xi^{1-\frac{1}{n}} 
      \left(\frac{A(t)\varphi'(\xi)}{B(t)}\right)^{q-1}.
   \end{align}
 Since
   \begin{align*}
     & -\xi B'(t)+n^p B^{1-\frac{1}{n}}(t)
      \xi^{1-\frac{1}{n}}\left(\frac{A(t)\varphi'(\xi)}
      {B(t)}\right)^{p-1}-\frac{(1-\delta)n^q\chi}{\sqrt{\delta+1}}
            \cdot B^{1-\frac{1}{n}}(t)\xi^{1-\frac{1}{n}} 
      \left(\frac{A(t)\varphi'(\xi)}{B(t)}\right)^{q-1}
     \\ \nonumber
     &=\left(\frac{A(t)\varphi'(\xi)}{B(t)}\right)^{q-1}
      \Biggl\{-\xi B'(t)\left(\frac{A(t)\varphi'(\xi)}{B(t)}
      \right)^{1-q}
      \\ \nonumber
      &\quad\,+\left(n^p\left(\frac{A(t)\varphi'(\xi)}{B(t)}
      \right)^{p-q}-\frac{(1-\delta)n^q\chi}{\sqrt
      {\delta+1}}\right) B^{1-\frac{1}{n}}(t)
      \xi^{1-\frac{1}{n}}\Biggr\}
     \end{align*}   
 holds, a combination of \eqref{def;c3}, \eqref{ineq;Pwin9} and  
 Lemma  \ref{tool4;intermediate}, along with the relation $p\leq q$ yields that
   \begin{align}\label{est6}
      &\frac{B(t)}{A(t)\varphi'(\xi)}\cdot(\mathcal{P} 
     w_{\rm in})(s,t)
     \\ \nonumber
     &\leq \left(\frac{A(t)\varphi'(\xi)}{B(t)}\right)^{q-1}
      \Biggl\{-\xi B'(t)\sigma^{1-q}+\left(n^p\sigma^{p-q}-
      \frac{(1-\delta)n^q\chi}{\sqrt{\delta+1}}\right)B^{1-
      \frac{1}{n}}(t)\xi^{1-\frac{1}{n}}\Biggr\}
      \\ \nonumber
      &=  \left(\frac{A(t)\varphi'(\xi)}{B(t)}\right)^{q-1}
      \sigma^{1-q}\Biggl\{-\xi B'(t)-n^q\sigma^{q-1}
      \Biggl(\frac{(1-\delta)\chi}{\sqrt{\delta+1}}
      -(n\sigma)^{p-q}\Biggr)B^{1-\frac{1}{n}}(t)
      \xi^{1-\frac{1}{n}}\Biggr\}
      \\ \nonumber
      &= \left(\frac{A(t)\varphi'(\xi)}{B(t)}\right)^{q-1}
      \sigma^{1-q}\left(-\xi B'(t)-\sigma^{q-1}c_1 
      B^{1-\frac{1}{n}}(t)\xi^{1-\frac{1}{n}}\right)
 \end{align}
 for all $t\in(0,T)$ and all $s\in(B(t),K\sqrt{B(t)})$. 
 Finally, according to that 
 \[
 \xi < \frac{K}{\sqrt{B(t)}}
 \] 
 which means that $s < K\sqrt{B(t)}$, 
 we infer from \eqref {condi;B6} and \eqref{def;kappan} that
 \begin{align}\label{est7}
   -\xi B'(t)-\sigma^{q-1}c_1 B^{1-\frac{1}{n}}(t)
   \xi^{1-\frac{1}{n}}&= \xi \cdot\left(-B'(t)-\sigma^{q-1}
   c_1 B^{1-\frac{1}{n}}(t)\xi^{-\frac{1}{n}}\right)
   \\ \nonumber
   &\leq \xi \cdot\left(-B'(t)-\sigma^{q-1}c_1 
   B^{1-\frac{1}{n}}(t)\cdot K^{-\frac{1}{n}}
   B^{\frac{1}{2n}}(t)\right)
   \\ \nonumber
   & = \xi \cdot\left(-B'(t)-\kappa_n B^{1-\frac{1}{2n}}
   (t)\right)
   \\ \nonumber
   &\leq 0
 \end{align}
 for all $t\in(0,T)$ and all $s\in(B(t),K\sqrt{B(t)})$. Thus 
 \eqref {est6} and \eqref{est7} 
 lead to the end of the proof.
 \end{proof}
\section{Blow-up. Proof of Theorem 1.1}\label{Sec4}

%
%

In this section, by virtue of a combination of 
Lemmas \ref{estimate;outer}, \ref{estimate;very inner}, 
\ref{estimate2;veryinner}, \ref{estimate1;intermediate} 
and \ref{estimate2;intermediate}, we can show our main 
purpose such that there is an initial data satisfying 
that the corresponding solution blows up by 
 using 
the comparison argument from \cite[Lemma 5.1]{BW}.

 \begin{proof}[{\bf Proof of Theorem \ref{mainthm1}}]
If $n=1$, 
thanks to \eqref{condi;mass}, for all $\chi > 1$ 
($\chi>0$ when $q>p$) and some $\lambda \in 
 (\frac{5-\sqrt{17}}{2},1)$, Lemma \ref{estimate1;intermediate} 
 entails that there exist $K>0$, $\kappa_1 >0$
 and $B_{01}\in (0,1)$ with properties listed there.
 On the other hand, if $n\geq 2$, $m > 0$ and $\lambda=\frac{1}{3}$, 
 then from the condition 
 \eqref{condi;chi} we can use 
 Lemma \ref{estimate2;intermediate} to see that for all 
 $\chi>0$ there exist $K>1, \kappa_n >0$ 
 and $B_{0n}\in (0,1)$ with properties noted there.
  Moreover, we 
  introduce $\kappa \in (0,\kappa_n]$ given by
 \begin{align}\label{condi;kappa1}
          \kappa\leq\frac{a_\lambda^{q-1}(nm)^q\chi K}
          {2(a_\lambda+b_\lambda)
          (K^2+a_\lambda R^n)^{q-1}\omega_n^qR^n\sqrt{1+K^{\frac{2}{n}-2}
          \cdot \frac{m^2}{\omega_n^2}}}
 \end{align}
and
 \begin{align}\label{condi;kappa2}
          \kappa\leq\frac{d}{\sqrt{d^2+1}}\cdot 
          \left(\dfrac{2\lambda A_T}{e^d}\right)^{p-1}
 \end{align}
as well as
 \begin{align}\label{condi;kappa3}
          \kappa\leq\frac{n^q \chi }{\sqrt{2}}\cdot
          \left(\dfrac{2\lambda A_T}{e^d}\right)^{q-1}
 \end{align}
with $d$ and $A_T$ as in \eqref{def;d} and \eqref{def;AT}, 
respectively, and we take $B_0\in(0,B_{0n}]$ satisfying 
\eqref{condi;B2}, \eqref{condi;B_0_2} and \eqref{condi;B7}.
Here we define
 \[\label{def;B}
         B(t) := \left(B_0^{\frac{1}{2n}}-\frac{\kappa}{2n}
         t\right)^{2n}
 \]
for $t\in(0,T)$ 
 with
 \[\label{def;T}
         T := \frac{2n}{\kappa}\cdot B_0^{\frac{1}{2n}}.
 \]
 Then, $B\in C^1([0,T))$ 
 is the solution of the following 
 differential equation\/{\rm :}
\begin{align}\label{condi;B}
    	\begin{cases}
          B'(t)= -\kappa B^{1-\frac{1}{2n}}(t),
         \\[1mm] 
          B(0)= B_0
         \end{cases}
 \end{align}
 for all $t\in(0,T).$ 
  According to Lemma \ref{def;subsolution}, putting
 \[
         \underline{w}(s,t) :=
           \begin{cases}
              w_{\rm in}(s,t)
              &\mbox{if}\ t\in[0,T)\ \mbox{and}\ 
              s\in[0,K\sqrt{B(t)}],
              \\[1mm] 
              w_{\rm out}(s,t)
              &\mbox{if}\ t\in[0,T)\ \mbox{and}\ 
              s\in(K\sqrt{B(t)},R^n]
           \end{cases}
 \]
with functions $w_{\rm in}$ and $w_{\rm out}$ as in 
\eqref{def;win} and \eqref{def;wout}, respectively, 
we see that the function $\underline{w}$ is well-defined and satisfies 
\[
\underline{w}\in C^1([0,R^n]\times[0,T))
\]
and $\underline{w}(\cdot,t)\in C^2([0,R^n]
\setminus \{B(t),K\sqrt{B(t)}\})$
for all $t\in[0,T)$. Moreover, thanks to \eqref{condi;kappa1}, \eqref{condi;kappa2}, 
\eqref{condi;kappa3}, \eqref{condi;B} and the fact that $B(t)\leq B_0 <1$, 
we can use Lemmas \ref{estimate;outer}, 
\ref{estimate;very inner}, \ref{estimate2;veryinner}, 
\ref{estimate1;intermediate} and \ref{estimate2;intermediate} 
to lead to functions $w_{\rm out}$ and $w_{\rm in}$ fulfilling
 \begin{align}\label{Pwout<0}
         (\mathcal{P}w_{\rm out})(s,t) \leq 0
 \end{align}
for all $t\in(0,T)$ and all $s\in(K\sqrt{B(t)},R^n)$ as well as
 \begin{align}\label{Pwin<0}
         (\mathcal{P}w_{\rm in})(s,t) \leq 0
 \end{align}
for all $t\in(0,T)$ and all $s\in(0,B(t))\cup(B(t),K\sqrt{B(t)})$.
Therefore, we obtain that
 \[
          (\mathcal{P}\underline{w})(s,t) \leq 0
 \]
for all $t \in(0,T)$ and all $s\in(0,R^n)\setminus\{B(t),K\sqrt{B(t)}\}$ 
since \eqref{Pwout<0} and \eqref{Pwin<0} hold. Here 
we assume that $u_0$ satisfies  \eqref{condi;ini1} and also
 \[
          \int_{B_r(0)}{u_0(x)\,dx}\geq M_m(r) := 
          \omega_n\underline{w}(r^n,0)
 \]
for all $r\in[0,R]$. Then, we can see that
 \[
         w(s,0)\geq\underline{w}(s,0)
 \]
 for all $s\in(0,R^n)$ with the solution $w$ of \eqref{lem;base} 
 defined by \eqref{def;w}. Moreover, putting $\widetilde{T} := 
\mbox{min}\{\tmax,T\}$, we derive that
 \[
         w(0,t)=\underline{w}(0,t)=0\quad \mbox{and}\quad        
         w(R^n,t)=\underline{w}(R^n,t)=\frac{m}{\omega_n}
 \]
for all $t\in(0,\widetilde{T})$. 
In order to use the 
comparison principle stated in \cite[Lemma 5.1]{BW}
we write $\alpha := 2-\frac{2}{n}\geq 0$ and let
 \[
     \phi(s,t,y_0,y_1,y_2) := n^{p+1}\cdot
     \frac{s^{\alpha}{y_1}^p
     y_2}
     {\sqrt{y_1^2+n^2s^{\alpha}
     y_2^2}}+n^q\chi\cdot
     \frac{(y_0-\frac{\mu}{n}s){y_1}^q}
     {\sqrt{1+s^{-\alpha}
     (y_0-\frac{\mu}{n}s)^2}}
 \]
for $(s,t,y_0,y_1,y_2)\in G := (0,R^n)\times(0,\infty)\times 
\mathbb{R}\times(0,\infty)\times\mathbb{R}$. Then  
$\phi\in C^1(G)$ with
 \begin{align}\label{5.1}
        \frac{\pa\phi}{\pa y_2}(s,t,y_0,y_1,y_2) 
        = \frac{n^{p+1}s^{\alpha}y_1^{p+2}}
        {{\sqrt{y_1^2+n^2s^\alpha y_2^2}}^3}\geq 0
 \end{align}
 and
 \begin{align*}
         \frac{\pa\phi}{\pa y_1}(s,t,y_0,y_1,y_2) &= n^{p+1}
         \cdot\frac{ps^{\alpha}y_1^{p-1}
         y_2(y_1^2+n^2s^{\alpha}y_2^2)-s^{\alpha}
         y_1^{p+1}y_2}{{\sqrt{y_1^2+n^2s^\alpha y_2^2}}^3}
         \\
        &\quad\,+n^q\chi\cdot\frac{q(y_0-\frac{\mu}{n}
        s)y_1^{q-1}}{\sqrt{1+s^{-\alpha}(y_0-\frac{\mu}{n}
        s)^2}}
 \end{align*}
 as well as
 \begin{align*}
        \frac{\pa\phi}{\pa y_0}(s,t,y_0,y_1,y_2) = n^q\chi
        \frac{y_1^q}{{\sqrt{1+s^{-\alpha}(y_0-\frac{\mu}
        {n}s)^2}}^3}
 \end{align*}
for all $(s,t,y_0,y_1,y_2)\in G $. 
Here if $S\subset(0,R^n)
\times \mathbb{R}\times(0,\infty)\times\mathbb{R}$ 
is compact, then there exist $ a_1,a_2 \in(0,\infty)$ and 
$\ b_1,b_2\in\mathbb{R}$ 
 satisfying that
 \[
        S\subset(0,R^n)\times \mathbb{R}\times(a_1,a_2)
        \times(b_1,b_2).
 \]
Therefore, putting $b := \mbox{max}\{|b_1|,|b_2|\}$ and 
$a := \mbox{max}\{a_1^{p-2},a_2^{p-2}\}$, 
we establish that for any $t\in (0,\widetilde{T})$,
 \begin{align}\label{5.3}
         \left|\frac{\pa\phi}{\pa y_1}(s,t,y_0,y_1,y_2)\right| 
         &\leq \left|n^{p+1}
         \cdot\frac{ps^{\alpha}y_1^{p-1}
         y_2(y_1^2+n^2s^{\alpha}y_2^2)-s^{\alpha}
         y_1^{p+1}y_2}{{\sqrt{y_1^2+n^2s^\alpha y_2^2}}^3}\right|
         \\ \nonumber
        &\quad\,+\left|n^q\chi\cdot\frac{q(y_0-\frac{\mu}{n}
        s)y_1^{q-1}}{\sqrt{1+s^{-\alpha}(y_0-\frac{\mu}{n}
        s)^2}}\right|
        \\ \nonumber
        &\leq \left|n^{p+1}\cdot\frac{(p-1)s^{\alpha}y_1^{p
         +1}y_2}{{\sqrt{y_1^2}}^3}\right|+\left|n^{p+3}\cdot
         \frac{ps^{2\alpha} y_1^{p-1}y_2^3}{{\sqrt{n^2s^\alpha y_2^2}}
         ^3}\right|
         \\ \nonumber
         &\quad\,+\left|n^q\chi qs^{\frac{\alpha}{2}}
         y_1^{q-1}\cdot\frac{\sqrt{s^{-\alpha}(y_0-\frac{\mu}{n}
         s)^2}}{\sqrt{1+s^{-\alpha}(y_0-\frac{\mu}{n}s)^2}}
         \right|
         \\ \nonumber
        &\leq \left|n^{p+1}(p-1)s^{\alpha}y_1^{p
         -2}y_2\right|+\left|n^p
         ps^\alpha y_1^{p-1}\right|+\left|n^q\chi qs^{\frac{\alpha}{2}}
         y_1^{q-1}
         \right|
         \\ \nonumber
        &\leq \left|n^{p+1}(p-1)R^{n\alpha}a^{p
         -2}b\right|+\left|n^p
         pR^{n\alpha} a_2^{p-1}\right|+\left|n^q\chi qR^{\frac{n\alpha}{2}}
         a_2^{q-1}
         \right|
 \end{align}
for all $(s,y_0,y_1,y_2)\in S$, and thus 
$|\frac{\pa\phi}{\pa y_1}
(\cdot,t,\cdot,\cdot,\cdot)|\in L_{\rm loc}^\infty((0,R^n)
\times\mathbb{R}\times(0,\infty)\times\mathbb{R})$, 
and for all $T_0 \in (0,\widetilde{T})$ and all $\Lambda>0$ 
we derive that 
 \begin{align}\label{5.2}
          \left|\frac{\pa\phi}{\pa y_0}(s,t,y_0,y_1,y_2)
          \right| \leq n^q\chi y_1^q \leq n^q\chi \Lambda^q
 \end{align}
for all $(s,t,y_0,y_1,y_2)\in G $ with $t\in(0,T_0)$ and $y_1\in(0,\Lambda)$.  
Thanks to \eqref{5.1}, \eqref{5.3} and \eqref{5.2}, 
we can use \cite[Lemma 5.1]{BW} to yield that
 \[
       w(s,t)\geq \underline{w}(s,t)
 \]
for all $s\in[0,R^n]$ and all $t\in[0,\widetilde{T})$. Since 
$w(0,t) = \underline{w}(0,t) = 0$ for all $t\in[0,\widetilde{T})$, 
the mean value theorem  
implies 
that for each $t\in[0,\widetilde{T})$ 
there is some $\theta(t)\in(0,R^n)$ with 
the property that
 \[
       w_s(\theta(t),t)=\frac{w(B(t),t)}{B(t)}\geq 
       \frac{\underline{w}(B(t),t)}{B(t)}=\frac{A(t)\varphi(1)}
       {B(t)}=\lambda\cdot \frac{A(t)}{B(t)}
 \]
for all $t\in[0,\widetilde{T})$. Noting 
that $u(r,t) = nw_s(r^n,t)$ 
for all $r\in (0,R)$ and all $t\in(0,\widetilde{T})$, we can show that
 \[
       \sup_{r\in(0,R)}{u(r,t)}\geq w_s(\theta(t),t)=\lambda
       \cdot \frac{A(t)}{B(t)}
 \]
for all $t\in(0,\widetilde{T})$. Thanks to the facts that 
\[
       B(t)\searrow 0\quad \mbox{and} \quad A(t)\to 
       \frac{m}{\omega_n}
       \cdot\frac{K^2}{K^2+a_\lambda R^n}
\]
as $t\nearrow T$ by \eqref{def;A}, using a consequence 
of the extensibility 
criterion \eqref{local}, we can see 
that \eqref{blow-up} holds with $T^\ast=\tmax \leq T < \infty$,
 which enables us to attain 
Theorem \ref{mainthm1}.
 \end{proof}  



\newpage
\begin{thebibliography}{99}
\bibitem{BW0}
N. Bellomo, M. Winkler, 
 {\it A degenerate chemotaxis system with flux limitation: 
 maximally extended solutions and absence of gradient blow-up}, 
 Comm.\  Partial Differential Equations {\bf 42} (2017), 436--473.
\bibitem{BW}
N. Bellomo, M. Winkler, 
{\it Finite-time blow-up in a degenerate chemotaxis system 
with flux limitation},
Trans.\  Amer.\  Math.\  Soc.\  Ser.\  B {\bf 4} (2017), 31--67.
\bibitem{C}
X. Cao,
{\it Global bounded solutions of the higher-dimensional 
Keller--Segel system under smallness conditions in optimal spaces},
Discrete Contin.\ Dyn.\ Syst.\ {\bf 35} (2015), 1891--1904.
\bibitem{HIY}
T. Hashira, S. Ishida, T. Yokota, {\it Finite-time blow-up for 
quasilinear degenerate Keller--Segel systems of 
parabolic--parabolic type}, J. Differential Equations {\bf 264} (2018),  6459--6485.  
\bibitem{HP}
T. Hillen, K. Painter,
 {\it A user's guide to PDE models for chemotaxis}, 
J. Math.\ Biol.\ {\bf 58} (2009), 183--217.
\bibitem{HW}
D. Horstmann, G. Wang, {\it Blow-up in a chemotaxis model 
without symmetry
assumptions}, Eur.\ J.\ Appl.\ Math.\ {\bf 12} (2001), 159--177.
\bibitem{HWI}
D. Horstmann, M. Winkler, {\it Boundedness vs. blow-up in a chemotaxis system}, J.\ Differential Equations {\bf 215} (2005), 52--107. 
\bibitem{IY}
S. Ishida, T. Yokota, {\it Global existence of weak solutions to 
quasilinear degenerate Keller--Segel systems of 
parabolic--parabolic type with small data}, 
J. Differential Equations {\bf 252} (2012), 2469--2491.
\bibitem{JL}
W. J\"{a}ger, S. Luckhaus, 
{\it On explosions of solutions to a system of 
partial differential equations modelling chemotaxis},  
Trans.\  Amer.\  Math.\  Soc.\  {\bf 329} (1992), 819--824.  
\bibitem{KS}
 E. F. Keller, L. A. Segel, {\it Traveling bands of 
 chemotactic bacteria: A theoretical
analysis}, J. Theor.\ Biol.\ {\bf 30} (1971), 235--248.
\bibitem{LM}
P. Lauren\c{c}ot, N. Mizoguchi,
 {\it Finite time blowup for the parabolic--parabolic 
 Keller--Segel system with critical diffusion}, 
Ann.\ Inst.\ H. Poincar\'{e} Anal.\ Non Lin\'{e}aire {\bf 34} (2017), 197--220. 
\bibitem{M}
Y. Mimura,
{\it The variational formulation of the fully parabolic 
Keller--Segel system with degenerate diffusion}, 
J. Differential Equations {\bf 263} (2017), 1477--1521.
\bibitem{MW}
N. Mizoguchi, M. Winkler, {\it Blow-up in the 
two-dimensional 
parabolic Keller--Segel 
system}, preprint.
\bibitem{OMY}
M. Mizukami, T. Ono, T. Yokota, {\it Extensibility criterion ruling out 
gradient blow-up in a quasilinear degenerate chemotaxis system}, 
preprint. 
\bibitem{NSY}
T. Nagai, T. Senba, K. Yoshida,
{\it Application of the Trudinger--Moser inequality 
to a parabolic system of chemotaxis}, 
Funkcial.\ Ekvac.\ {\bf 40} (1997), 411--433.
\bibitem{OY}
K. Osaki, A. Yagi, {\it Finite dimensional attractor for 
one-dimensional Keller--Segel equations}, 
Funkcial.\ Ekvac.\ {\bf 44} (2001), 441--469.
\bibitem{W}
M. Winkler,
{\it Finite-time blow-up in the higher-dimensional 
parabolic--parabolic Keller--Segel system}
J. Math.\ Pures Appl.\ {\bf 100} (2013), 748--767.
 \end{thebibliography}
\end{document}